\documentclass[preprint,10pt]{elsarticle}
\usepackage{amssymb}
\usepackage{atbegshi}
\AtBeginDocument{\AtBeginShipoutNext{\AtBeginShipoutDiscard}}
\usepackage[english]{babel}
\usepackage[utf8]{inputenc} 
\usepackage{times}
\usepackage[babel]{csquotes}
\usepackage{amsthm}
\usepackage{lmodern}
\usepackage[a4paper]{geometry}
\usepackage{epstopdf}
\usepackage[table]{xcolor}
\usepackage{mathrsfs}
\usepackage{latexsym,amssymb,amsfonts}
\usepackage{enumerate}
\usepackage{graphicx}
 \usepackage{url}

\usepackage{float}
\usepackage{graphicx}
\usepackage{subfigure}
\usepackage{alltt}
\usepackage{amsfonts}
\usepackage{amsmath}
\usepackage[labelfont=bf]{caption}
\usepackage[colorlinks]{hyperref}
\AtBeginDocument{
\hypersetup{
    colorlinks=true,
    citecolor=red,
    filecolor=red,
    linkcolor=red,
    urlcolor=red
}}
\def\dis{\displaystyle}

\newtheorem{thm}{Theorem}[section]

\newtheorem{lem}[thm]{Lemma}

\theoremstyle{definition}
\newtheorem{defn}[thm]{Definition}

\newtheorem{rem}[thm]{Remark}

\def\R{\mathbb{R}}

\def\P{\mathbb P}

\def\dt{\textup{d}}

\def\dt{\textup{d}}
\journal{00}
\usepackage{geometry}
\usepackage{dsfont}
\begin{document}
\selectlanguage{english}
\begin{frontmatter}

\title{The long-time behaviour of a stochastic SIR epidemic model with distributed delay and multidimensional  Lévy jumps}
\author{Driss Kiouach\footnote{Corresponding author.\\
\hspace*{0.4cm}E-mail addresses: \href{d.kiouach@uiz.ac.ma}{d.kiouach@uiz.ac.ma} (D. Kiouach), \href{yassine.sabbar@usmba.ac.ma}{yassine.sabbar@usmba.ac.ma} (Y. Sabbar).} and Yassine Sabbar}
\address{LPAIS Laboratory, Faculty of Sciences Dhar El Mahraz, Sidi Mohamed Ben Abdellah University, Fez, Morocco.}   
\vspace*{1cm}

\begin{abstract}
Recently\textcolor{black}{,} emerging epidemics like COVID-19 and its variants require predictive mathematical models to implement suitable responses in order to limit their negative and profound impact on society. The SIR (Susceptible-Infected-Removed) system is a straightforward mathematical formulation to model the dissemination of many infectious \textcolor{black}{diseases}. The present paper reports novel theoretical and analytical results for a perturbed version of an SIR model with Gamma-distributed delay. Notably, our epidemic model is represented by Itô-Lévy stochastic differential equations in order to simulate sudden and unexpected external phenomena. By using some new and ameliorated mathematical approaches, we study the long-run characteristics of the perturbed delayed model. Within this scope, we give sufficient conditions for two interesting asymptotic proprieties: extinction and persistence of the epidemic. One of the most interesting results is that the dynamics of the stochastic model are closely related to the intensities of white \textcolor{black}{noises} and Lévy jumps, which can give us a good insight into the evolution of the epidemic in some unexpected situations. Our work complements the results of some previous investigations and provides a new approach to predict and analyze the dynamic behavior of epidemics with distributed delay. For illustrative purposes, numerical examples are presented for checking the theoretical study. \vskip 2mm


\end{abstract}
\end{frontmatter}
\
\section{Introduction}
Epidemic models under various formal frameworks are an auxiliary tool to acquire information about the dynamics of epidemic transmission and the impact of different intervention strategies \cite{1}. Recently, the employment  of these models to generate long-term epidemic forecasts is increased with the rising number of emerging and re-emerging epidemic outbreaks. For example, COVID-19 and its new variant identified in England present an urgent and serious challenge due to their contagious nature and frequently changing characteristics \cite{covid,Q1,Q2,salih2,salih3}. A mathematical model for measuring and preventing the continued spread of COVID-19 is firmly required to understand the mechanisms of its dynamic and predict its future. Considerable attention has been paid to the analysis of susceptible-infectious-removed (SIR) type model, which is proposed to describe the dissemination of COVID-19 before suggesting more adapted and complex epidemic models. In the SIR model, we often \textcolor{black}{assume} that recovered individuals can get continuous immunity \cite{2}. Many studies have paid close attention to the characteristics of the long-term epidemics immune response \cite{3,4,5,yan4}. To confer the realistic aspect of the epidemic model and make it biologically reasonable, numerous scholars considered the SIR epidemic model with time delay because an individual may not be infectious until some time after becoming infected \cite{delay1,2}. In the \textcolor{black}{above-mentioned} works, the time delay is assumed to be single-valued. The constant delay may be considered if the variation of the time is known exactly, which is not real for many biological \textcolor{black}{reasons} \cite{d1}. Considering the variable infectivity in the time interval yields a model with a distributed delay \cite{d2}. Therefore, it is more realistic to introduce a continuously distributed delay in the biological modeling \cite{d3,d4}. Analyzing the characteristics of the SIR model with a distributed time delay still a rich subject that may deliver new comprehension of \textcolor{black}{the epidemics propagation} which motivates our work. According to the approach of Muroya et al. \cite{m1}, we consider the delay kernel $\mathcal{G}:[0,\infty)\to[0,\infty)$ as a normalized $\mathbf{L}^1$-function, \textcolor{black}{i.e,} $\int_0^{\infty}\mathcal{G}(s)\textup{d} s=1$.  The average delay for the kernel $\mathcal{G}$ can be presented by the following quantity $\int_0^{\infty}s\mathcal{G}(s)\textup{d} s<\infty$. Hence, \textcolor{black}{the incidence rate} at time $\tau$ can be presented as the following form: $\beta S(\tau)\int_{-\infty}^{\tau}\mathcal{G}(\tau-s)I(s)\textup{d} s$, where $\beta$ denotes the transmission rate, $S ( t )$ and $I ( t )$ represent the fractions of susceptible and infective individuals at time $t$. The SIR epidemic model with distributed delay can be expressed as follows \cite{da}:
 \begin{flalign}
&\begin{cases}
\dis\textup{d} S(t)=\left(A-\mu_1 S(t)-\beta S(t)\int_{-\infty}^t\mathcal{G}(t-s)I(s)\textup{d} s \right)\textup{d} t,\\
\dis\textup{d} I(t)=\left(\beta S(t)\int_{-\infty}^t\mathcal{G}(t-s)I(s)\textup{d} s-(\mu_2+\gamma)I(t)\right)\textup{d} t,\\
\textup{d} R(t)=\big(\gamma I(t)-\mu_3 R(t)\big)\textup{d} t,
\end{cases}&
\label{s1}
\end{flalign}
where $R ( t )$ is the fraction number of recovered populations at time $t$. The remaining parameters appearing in this system are described as follows:
\begin{itemize}
\item[$\bullet$] $A$ is the recruitment rate of susceptible individuals corresponding to births and immigration.
\item[$\bullet$] $\mu_1$, $\mu_3$ are the natural death rates \textcolor{black}{associated respectively to the susceptible and recovered populations,} $\mu_2$ is a general mortality rate including the effect of the disease fatality.
\item[$\bullet$] $\gamma$ is the rate of individuals leaving $I$ to $R$ (recovered rate).
\end{itemize}

The threshold number of the deterministic system (\ref{s1}) is $\dis\mathcal{T}^{\star}=\frac{\beta A}{\mu_1(\mu_2+\gamma)}$  which determines the persistence ($\mathcal{T}^{\star}>1$) or the extinction ($\mathcal{T}^{\star}<1$) of the epidemic. Many studies showed that the deterministic epidemic model (\ref{s1}) is suitable to describe the transmission process of some known epidemics such as Rubella, Whooping cough, Measles and Smallpox. Due to many biological and mathematical considerations \cite{da}, in this paper, we consider the delay kernel with Gamma distribution $\dis\mathcal{G}(s)=\frac{s^n\eta^{n+1}e^{-\eta s}}{n!},\hspace*{0.3cm}s\in(0,\infty)$, where the constant $\eta>0$ is the rate of exponential fading memory, which means the retrogradation of the past memories effect. In this paper, we consider the low kernel function $\mathcal{G}$ with $n=0$. By letting  $\mathcal{D}(t)=\int_{-\infty}^t\eta e^{-\eta(t-s)}I(s)\textup{d}s$ and using the linear chain approach, system (\ref{s1}) can be transformed into the following equivalent system:
\begin{flalign}
&\begin{cases}
\textup{d} S(t)=\big(A-\mu_1 S(t)-\beta S(t)\mathcal{D}(t)\big)\textup{d} t,\\
\textup{d} I(t)=\big(\beta S(t)\mathcal{D}(t)-(\mu_2+\gamma)I(t)\big)\textup{d} t,\\
\textup{d} R(t)=\big(\gamma I(t)-\mu_3 R(t)\big)\textup{d} t,\\
\textup{d}\mathcal{D}(t)=\eta\big(I(t)-\mathcal{D}(t)\big)\textup{d} t.
\end{cases}&
\label{s1d}
\end{flalign}

Although the use of deterministic models can explain and simulate some phenomena in real life, such models do not actually consider the effect of the natural stochasticity, and we plainly wish to learn how randomness affects our epidemic models \cite{yan1,yan2,yan3,ls2,salih1,mp11,mp12}. Generally, one of the ordinary extensions from the deterministic SIR model to the stochastic version is to incorporate environmental white \textcolor{black}{noises, which appear} from an almost continuous series of small variations on the \textcolor{black}{model} parameters \cite{s1,sis30,sis31,ls,ls1,mp,ls3,ls4}. Therefore, the stochastic delayed SIR epidemic can be an accurate tool to predict the long-run dynamics of infectious epidemics \cite{s2,s3,51,52,53,54,55,zahri1,zahri2}. In \cite{da}, the authors inserted the stochastic perturbation in the model (\ref{s1}) by assuming that the white noise is directly proportional to the variable $S$ and  they obtained the following stochastic system:
  \begin{flalign}
&\begin{cases}
\textup{d} S(t)=\big(A-\mu_1 S(t)-\beta S(t)\mathcal{D}(t)\big)\textup{d} t+\sigma S(t)\textup{d} \mathcal{W}(t),\\
\textup{d} I(t)=\big(\beta S(t)\mathcal{D}(t)-(\mu_2+\gamma)I(t)\big)\textup{d} t,\\
\textup{d} R(t)=\big(\gamma I(t)-\mu_3 R(t)\big)\textup{d} t,\\
\textup{d}\mathcal{D}(t)=\eta\big(I(t)-\mathcal{D}(t)\big)\textup{d} t.
\end{cases}&
\label{s12}
\end{flalign}
where $\mathcal{W}(t)$ is a standard Brownian motion with associated intensity $\sigma>0$. Specifically, they proved the existence and uniqueness of an ergodic stationary distribution to the model (\ref{s12}). Then, they established sufficient conditions for the extinction of a disease \textcolor{black}{that spreads according to this model}. On the basis of these findings, one question was catches our attention. It is possible to develop and generalize the stochastic model proposed in \cite{da}?. So, the objective of this work is to expound on this problem and provide a suitable analytical context. Specifically, we aim to describe the strong fluctuations by considering a general version of the dynamical model (\ref{s12}). It is clear that the population systems may suffer certain sudden environmental catastrophes, such as earthquakes, floods, droughts, etc \cite{sis11,sis12,sis14}. For example, the recent massive explosion in the port city of Beirut. The impact of this unexpected disaster has been extremely devastating, especially when it has occurred simultaneously with the \textcolor{black}{ COVID-19} pandemic. This led to a sudden worsening of the health situation and a jump increase in the number of deaths. Mathematically, we \textcolor{black}{use} the Lévy process to describe the phenomena that cause a big jump to occur occasionally \cite{lsbook,sadik1,sadik2,56,57}. By considering this type of random perturbations, the model \eqref{s1d} becomes the following system of stochastic differential equations with Lévy jumps (SDE-Js for short):
\begin{flalign}
&\begin{cases}
\dis\dt S(t)=\big(A-\mu_1 S(t)-\beta S(t)\mathcal{D}(t)\big)\textup{d} t+ \sigma_1 S(t) \dt \mathcal{W}_1(t)+\int_{\mathcal{U}} \lambda_1(u)S(t^{-})\widetilde{\mathcal{N}}(\textcolor{black}{\textup{d}t},\textup{d}u),\\[6pt]
\dis\dt I(t)=\big(\beta S(t)\mathcal{D}(t)-(\mu_2+\gamma)I(t)\big)\textup{d} t+ \sigma_2 I(t) \dt \mathcal{W}_2(t)+\int_{\mathcal{U}} \lambda_2(u)I(t^{-})\widetilde{\mathcal{N}}(\textup{d}t,\textup{d}u),\\
\dis\dt R(t)=\big(\gamma I(t)-\mu_3 R(t)\big)\textup{d} t+ \sigma_3 R(t) \dt \mathcal{W}_3(t)+\int_{\mathcal{U}} \lambda_3(u)R(t^{-})\widetilde{\mathcal{N}}(\textcolor{black}{\textup{d}t},\textup{d}u),\\
\dis\dt \mathcal{D}(t)=\eta\big(I(t)-\mathcal{D}(t)\big)\textup{d} t+\sigma_4 \mathcal{D}(t) \dt \mathcal{W}_4(t)+\int_{\mathcal{U}} \lambda_4(u)\mathcal{D}(t^{-})\widetilde{\mathcal{N}}(\textcolor{black}{\textup{d}t},\textup{d}u),
\end{cases}&
\label{s13}
\end{flalign}
where $S(t^{-})$, $I(t^{-})$, $R(t^{-})$ and $\mathcal{D}(t^{-})$ are the left limits of $S(t)$, $I(t)$, $R(t)$ and $\mathcal{D}(t)$, respectively. $\mathcal{W}_{i}(t)$ $(i = 1, 2,3,4)$ are independent Brownian motions and $\sigma_{i}>0$ $(i = 1, 2,3,4)$ are their intensities. $\mathcal{N}$ is a Poisson counting measure with compensating  martingale $\widetilde{\mathcal{N}}$ and characteristic measure $\nu$ on a measurable subset $\mathcal{U}$ of $(0,\infty)$ satisfying $\nu(\mathcal{U})<\infty$. $\mathcal{W}_{i}(t)$ $(i = 1, 2,3,4)$ are independent of $\mathcal{N}$. We assumed that $\nu$ is a Lévy measure such that $\widetilde{\mathcal{N}}(\textcolor{black}{\textup{d}t},\textup{d}u)=\mathcal{N}(\textcolor{black}{\textup{d}t},\textup{d}u)-\nu(\textup{d}u)\textcolor{black}{\textup{d}t}$ and we suppose that the function $\lambda_i: Z\times\Omega\to\R$ is bounded and continuous. Throughout this paper, we let $(\Omega,\mathcal{F},\P)$ denotes a complete probability space with a filtration $\{\mathcal{F}_t\}_{t\geq 0}$ satisfying these conditions: right continuous and $\mathcal{F}_0$ contains all $\P$-null sets. We also assume that $\mathcal{W}_i(t)$ \textcolor{black}{is} defined on this probability space. Since \textcolor{black}{the compartment }$R(t)$ does not appear in the equations of  $S(t)$ $I(t)$ and $\mathcal{D}(t)$, it is sufficient to analyze the dynamic behavior of the following SDE-J model: 
\begin{flalign}
&\begin{cases}
\dis\dt S(t)=\big(A-\mu_1 S(t)-\beta S(t)\mathcal{D}(t)\big)\textup{d} t+ \sigma_1 S(t) \dt \mathcal{W}_1(t)+\int_{\mathcal{U}} \lambda_1(u)S(t^{-})\widetilde{\mathcal{N}}(\textcolor{black}{\textup{d}t},\textup{d}u),\\[6pt]
\dis\dt I(t)=\big(\beta S(t)\mathcal{D}(t)-(\mu_2+\gamma)I(t)\big)\textup{d} t+ \sigma_2 I(t) \dt \mathcal{W}_2(t)+\int_{\mathcal{U}} \lambda_2(u)I(t^{-})\widetilde{\mathcal{N}}(\textup{d}t,\textup{d}u),\\
\dis\dt \mathcal{D}(t)=\eta\big(I(t)-\mathcal{D}(t)\big)\textup{d} t+\sigma_4 \mathcal{D}(t) \dt \mathcal{W}_4(t)+\int_{\mathcal{U}} \lambda_4(u)\mathcal{D}(t^{-})\widetilde{\mathcal{N}}(\textcolor{black}{\textup{d}t},\textup{d}u).
\end{cases}&
\label{s2}
\end{flalign}

In this study, we develop a \textcolor{black}{new analysis} to deal with stochastic models with jumps in epidemiology. Our main goal is to investigate sufficient conditions of the stochastic extinction and persistence in the mean. These two important properties are sufficient to predict and analyze the dynamics of a given epidemic. We apply a new approach to estimate the values of the averages $\dis t^{-1}\int^t_0\psi(s)\textup{d}s$ and $\dis \textcolor{black}{t^{-1}}\int^t_0\psi^2(s)\textup{d}s$, where $\psi (t)$ is the positive solution of the following subsystem:
\begin{align}
\begin{cases}\dt\psi(t)=\big(A-\mu_1\psi(t)\big)\textcolor{black}{\textup{d}t}+\sigma_1\psi(t)\dt \mathcal{W}_1(t)+\dis\int_{\mathcal{U}}\lambda_1(u)\psi(t^{-})\tilde{\mathcal{N}}(\textcolor{black}{\textup{d}t},\textup{d}u), \hspace{0.5cm}\forall t>0,\\
\dis\psi(0)=S(0)>0.\end{cases}
\label{s3}
\end{align}
Our approach allows us to close the gap left by using the classical method presented for example in \cite{58}. Furthermore, we give an optimal sufficient condition for the stochastic extinction. For the purpose of well understanding the dynamics of the delayed model (\ref{s2}), we give a sufficient condition of the disease persistence. The analysis in this paper seems to be promising to investigate other related stochastic delayed models with Lévy \textcolor{black}{noises} in epidemiology and even in biology.

This work is organized as follows. In section \ref{sec1}, we verify the well-posedness of the stochastic model (\ref{s2}). In section \ref{sec2}, we give sufficient conditions for the extinction and the persistence in the mean of the disease. Finally, in section \ref{sec4}, numerical simulations are carried out to confirm the theoretical study.
\section{\textcolor{black}{Existence and uniqueness of the global positive solution}} \label{sec1}
To study the long-term properties of an infectious disease system, the prime concern is whether the solution is unique, positive and global \textcolor{black}{in time}. In this short section, motivated by the approach presented in \cite{ki}, we show the well-posedness of the stochastic model (\ref{s2}). Mainly, the primary key to treat the said problem is to construct a suitable Lyapunov function. According to some analytical and mathematical reasons, it is necessary that we make the following two hypotheses:
\begin{itemize}
\item \textbf{($\mathcal{H}_1$):} We assume that the jump coefficients $\lambda_i(u)$ in (\ref{s2}) satisfy $\dis\int_{\mathcal{U}} \lambda_i^2(u)\nu(\textup{d} u)<\infty$ $(i=1,2,4)$.
\item \textbf{($\mathcal{H}_2$):} For all $i=1,2,4$, we assume that  $1+\lambda_i(u)>0$  and $\dis\int_{\mathcal{U}}\Big( \lambda_i(u)-\ln \big(1+\lambda_i(u)\big)\Big)\nu(\textup{d} u)<\infty.$
\end{itemize}
By assumption \textbf{($\mathcal{H}_1$)}, the coefficients of the system (\ref{s2}) are locally Lipschitz continuous, then for any initial value $(S(0), I(0), \mathcal{D}(0))\in \R^{3}_{+}$ there is a unique local solution $(S(t), I(t), \mathcal{D}(t))$ on $t\in[0,\tau_{e})$, herein, $\tau_{e}$ represents the explosion time. In the following theorem, our goal is to show that the solution is positive and global. 
\begin{thm}
Let hypotheses \textbf{$\textup{(}\mathcal{H}_1\textup{)}$} and \textbf{$\textup{(}\mathcal{H}_2\textup{)}$} hold. For any initial value $(S(0),I(0),\mathcal{D}(0))\in \R^3_{+}$, there exists a unique positive solution $(S(t),I(t),\mathcal{D}(t))$ of the SDE-J (\ref{s2}) on $t\geq 0$ and the solution will remain in $\R^{3}_{+}$ with probability one, namely $(S(t),I(t),\mathcal{D}(t))\in \R^3_{+}$ for all $t\geq 0$ almost surely (a.s.)\textbf{.} 
\label{thmp}
\end{thm}
\begin{proof}
We only need to prove that $\tau_{e}=\infty$ almost surely. Let $\epsilon_0>0$ be sufficiently large such that each component of $(S(0), I(0), \mathcal{D}(0))$ all lies in the interval \textcolor{black}{$\dis\Big(\frac{1}{\epsilon_0},\epsilon_0\Big)$}. For each integer $\epsilon\geq \epsilon_0$, we define the following stopping time
\begin{align*}
\tau_{\epsilon}&=\inf \left\lbrace t\in [0,\tau_{e})|\; S(t)\not\in \Big(\frac{1}{\epsilon},\epsilon\Big),\;\textcolor{black}{\mbox{or}}\; I(t)\not\in \Big(\frac{1}{\epsilon},\epsilon\Big),\;\mbox{or}\;\mathcal{D}(t)\not\in \Big(\frac{1}{\epsilon},\epsilon\Big)\right\rbrace.
\end{align*}
Set $\inf \emptyset=\infty$ ($\emptyset$ denotes the  empty set) and let $\tau_{\infty}=\underset{\epsilon\to\infty}{\lim}\tau_{\epsilon}$. Evidently, $\tau_{\epsilon}$ is increasing as $\epsilon\to\infty$. Moreover, $\tau_{\infty}\leq \tau_{e}$. If we can prove that $\tau_{\infty}=\infty$ a.s., then $\tau_{e}=\infty$ and $(S(t),I(t),\mathcal{D}(t))\in \R^{3}_{+}$ for all $t\geq 0$ almost surely. Specifically, we need to show that $\tau_{\infty}=\infty$ a.s. By supposing the opposite, we can consider a pair of positive constants $T>0$ and $k\in(0,1)$ such that $\P\{\tau_{\infty}\leq T\}>k$.
Hence, there exists an integer $\epsilon_1\geq \epsilon_0$ such that
\begin{align}
\P\{\tau_{\epsilon}\leq T\}\geq k\hspace{0.3cm}\mbox{for all}\hspace{0.3cm} \epsilon> \epsilon_1.
\label{11}
\end{align}
Construct a $\mathcal{C}^2$-function $\mathcal{V}:\R^3_+\to \textcolor{black}{[0,+\infty)}$ by
\begin{align*}
\mathcal{V}(S,I,\mathcal{D})=\left(S-m-m\ln \frac{S}{m}\right)+(I-1-\ln I)+\frac{\mu_2+\gamma}{\eta}(\mathcal{D}-1-\ln \mathcal{D}),
\end{align*}
where $m>0$ is a positive constant to be determined later. Obviously, this function is non-negative which can be seen from $x-1-\ln x \textcolor{black}{\geq} 0$ for all $x> 0$. According to the general Itô's formula \cite{lev1}, we obtain for all $0\leq t<\tau_{\epsilon}$,
\begin{align*}
\textup{d} \mathcal{V}(S,I,\mathcal{D})&=\mathcal{L} \mathcal{V}(S,I,\mathcal{D})\textup{d} t+\left(1-\frac{m}{S}\right)\sigma_1 S \textup{d} \mathcal{W}_1(t)\\&\;\;\;+\left(1-\frac{1}{I}\right)\sigma_2 I \textup{d} \mathcal{W}_2(t)+\frac{\mu_2+\gamma}{\eta}\left(1-\frac{1}{ \mathcal{D}}\right)\sigma_4 \mathcal{D} \textup{d} \mathcal{W}_4(t)\\&\;\;\;+\int_{\mathcal{U}}\lambda_1(u)S-m \ln(1+\lambda_1(u))\widetilde{\mathcal{N}}(\textup{d} t,\textup{d} u)\\&\;\;\;+\int_{\mathcal{U}}\lambda_2(u)I- \ln(1+\lambda_2(u))\widetilde{\mathcal{N}}(\textup{d} t,\textup{d} u)\\&\;\;\;+\int_{\mathcal{U}}\frac{\mu_2+\gamma}{\eta}\Big( \lambda_4(u)\mathcal{D}- \ln(1+\lambda_4(u))\Big)
\widetilde{\mathcal{N}}(\textup{d} t,\textup{d} u),
\end{align*}
where,
\begin{align*}
\mathcal{L}\mathcal{V}(S,I,\mathcal{D})&=A-\mu_1 S-\frac{m A}{S}
+m \beta \mathcal{D}+m\mu_1-(\mu_2+\gamma)I-\frac{\beta S \mathcal{D}}{I}\\&\;\;\;+(\mu_2+\gamma)+(\mu_2+\gamma)I-(\mu_2+\gamma)\mathcal{D}
-\frac{\textcolor{black}{(}\mu_2+\gamma\textcolor{black}{)}\textcolor{black}{I}}{\mathcal{D}}+(\mu_2+\gamma)\\&\;\;\;+\frac{m\sigma_1^2}{2}+\frac{\sigma_2^2}{2}+\frac{\mu_2+\gamma}{\eta}\frac{\sigma_4^2}{2}+
\int_{\mathcal{U}}m\lambda_1(u)- m\ln(1+\lambda_1(u))\nu(\textup{d} u)\\&\;\;\;+\int_{\mathcal{U}}\lambda_2(u)- \ln(1+\lambda_2(u))\nu(\textup{d} u)+\int_{\mathcal{U}}\frac{\mu_2+\gamma}{\eta}\Big(\lambda_4(u)- \ln(1+\lambda_4(u))\Big)\nu(\textup{d} u).
\end{align*}
Then
\begin{align*}
\mathcal{L}\mathcal{V}(S,I,\mathcal{D})&\leq A+2(\mu_2+\gamma)+m\mu_1+\textcolor{black}{\big(m\beta -(\mu_2+\gamma)\big)}\mathcal{D}+\frac{m\sigma_1^2}{2}\\&\;\;\;+\frac{\sigma_2^2}{2}+\frac{\mu_2+\gamma}{\eta}\frac{\sigma_4^2}{2}+\int_{\mathcal{U}}m\lambda_1(u)- m\ln(1+\lambda_1(u))\nu(\textup{d} u)\\&\;\;\;+\int_{\mathcal{U}}\lambda_2(u)- \ln(1+\lambda_2(u))\nu(\textup{d} u)+\int_{\mathcal{U}}\frac{\mu_2+\gamma}{\eta}\Big(\lambda_4(u)- \ln(1+\lambda_4(u))\Big)\nu(\textcolor{black}{\textup{d} u}).
\end{align*}
Given the fact that $x-\ln(1+x)\geq0$ for all $x> 1$ and the hypothesis \textbf{($\mathcal{H}_2$)}, we define
\begin{align*}
\mathcal{J}_1&\equiv\int_{\mathcal{U}}m\lambda_1(u)- m\ln(1+\lambda_1(u))\nu(\textup{d} u)+\int_{\mathcal{U}}\lambda_2(u)- \ln(1+\lambda_2(u))\nu(\textup{d}t u)\\&\;\;\;+\int_{\mathcal{U}}\frac{\mu_2+\gamma}{\eta}\Big(\lambda_4(u)- \ln(1+\lambda_4(u))\Big)\nu(\textcolor{black}{\textup{d} u}).
\end{align*}
To simplify, we choose $m=\frac{\mu_2+\gamma}{\beta}$. Then, we obtain
\begin{align*}
\mathcal{L}\mathcal{V}(S,I,\mathcal{D})\leq A+2(\mu_2+\gamma)+m\mu_1+\frac{m\sigma_1^2}{2}+\frac{\sigma_2^2}{2}+\frac{\mu_2+\gamma}{\eta}\frac{\sigma_4^2}{2}+\mathcal{J}_1
\equiv \mathcal{J}_2.
\end{align*}
The proof of the remainder is similar to the proof of Theorem 2 in \cite{ki}, so we omitted it.
\end{proof}
\section{Conditions of stochastic extinction and permanence in the mean of the epidemic}\label{sec2}
In mathematical epidemiology, we are generally interested in two things, the first is to know when the epidemic will die out, and the second is when it will continue and persist. In this section, we will try our best to find sufficient conditions for these two interesting asymptotic proprieties in terms of model parameters and intensities of noises. For the sake of notational simplicity, we define
\begin{itemize}
\item[$\bullet$] $\dis\bar{\sigma}\triangleq\max\{\sigma_1^2,\sigma_2^2,\sigma_4^2\}$\hspace{0.3cm} and \hspace{0.3cm}$\vartheta\triangleq\min\{\mu_1,\mu_2+\gamma-\eta,\eta\}$.
\item[$\bullet$] $\bar{\lambda}(u)\triangleq\max\{\lambda_1(u),\lambda_2(u),\lambda_4(u)\}$ \hspace{0.3cm}and \hspace{0.3cm}$\underline{\lambda}(u)\triangleq\min\{\lambda_1(u),\lambda_2(u),\lambda_4(u)\}$.
\item[$\bullet$] $\hat{\zeta}_{p}(u)\triangleq\big[1+\bar{\lambda}(u)\big]^{p}-1-p\bar{\lambda}(u)$ \hspace{0.3cm}and\hspace{0.3cm} $\check{\zeta}_{p}(u)\triangleq\big[1+\underline{\lambda}(u)\big]^{p}-1-p\underline{\lambda}(u)$.
\item[$\bullet$] $\xi(u)\triangleq\max\big\{\hat{\zeta}_{p}(u),\check{\zeta}_{p}(u)\big\}$ \hspace{0.3cm}and \hspace{0.3cm}$\ell_{p}\triangleq\dis\int_{\mathcal{U}}\xi(u)\nu(\textup{d}u)$.
\end{itemize}
To properly study the long-term of the perturbed  model (\ref{s2}), we have the following additionally hypotheses on the jump-diffusion coefficients:
\begin{itemize}
\item[$\bullet$] \textbf{($\mathcal{H}_3$):} For $i=1,2,4$, we assume that $\dis\int_{\mathcal{U}} \Big[\ln(1+\lambda_{\textcolor{black}{i}}(u))\Big]^2\nu(\textup{d}u)<\infty$ .
\item[$\bullet$] \textbf{($\mathcal{H}_4$):} For $i=1,2,4$, we assume that $\dis\int_{\mathcal{U}} \Big[\big(1+\bar{\lambda}(u)\big)^{2}-1\Big]^2\nu(\textup{d}u)<\infty$.
\item[$\bullet$] \textbf{($\mathcal{H}_5$):} We suppose that there exists some real number $p>2$  such that $\dis\chi_{1,p}=\vartheta-\frac{(p-1)}{2}\bar{\sigma}-\frac{1}{p}\ell_{p}> 0.$ 
\end{itemize}
For the convenience of discussion in the stochastic model (\ref{s2}), we introduce two lemmas which will be used in our analysis. 
\begin{lem}[\cite{pr}]
We assume that the conditions \textbf{($\mathcal{H}_4$)} and \textbf{($\mathcal{H}_5$)} hold. Let $(S(t),I(t),\mathcal{D}(t))$ be the positive solution of the system (\ref{s2}) with any given initial condition $(S(0),I(0),\mathcal{D}(0))\in\R^3_+$. Let also $\psi(t)\in\R_+$ be the solution of the equation (\ref{s3}) with any given initial value $\psi(0)=S(0)\in\R_+$. Then
\begin{itemize}
\item $\dis \underset{t\to\infty}{\lim}~t^{-1}~\psi(t)=0, \hspace{0.3cm} \underset{t\to\infty}{\lim}~t^{-1}~\psi^2(t)=0, \hspace{0.3cm}\underset{t\to\infty}{\lim}~t^{-1}~S(t)=0,  \hspace{0.2cm} \underset{t\to\infty}{\lim}~t^{-1}~I(t)=0,  \hspace{0.2cm}\mbox{and}\hspace{0.2cm}  \underset{t\to\infty}{\lim}~t^{-1}~\mathcal{D}(t)=0 \hspace{0.3cm}\mbox{a.s.}$
\item $\dis \underset{t\to\infty}{\lim}~t^{-1}~\int^t_0 \psi(s)\dt \mathcal{W}_1(s)=0, \hspace{0.3cm}\underset{t\to\infty}{\lim}~t^{-1}~\int^t_0 \psi^2(s)\dt \mathcal{W}_1(s)=0, \hspace{0.3cm}\underset{t\to\infty}{\lim}~t^{-1}~\int^t_0 S(s)\dt \mathcal{W}_1(s)=0, \\\underset{t\to\infty}{\lim}~t^{-1}~\int^t_0 I(s)\dt \mathcal{W}_2(s)=0,  \hspace{0.2cm}\mbox{and}\hspace{0.2cm}  \underset{t\to\infty}{\lim}~t^{-1}~\int^t_0 \mathcal{D}(s)\dt \mathcal{W}_4(s)=0 \hspace{0.3cm}\mbox{a.s.}$
\item $\dis \underset{t\to\infty}{\lim}~t^{-1}~\int^t_0\int_{\mathcal{U}} \lambda_1(u)\psi(s^{-})\widetilde{\mathcal{N}}(\textup{d}s,\textup{d}u)=0,\hspace{0.3cm}\underset{t\to\infty}{\lim}~t^{-1}~\int^t_0\int_{\mathcal{U}} \Big((1+\lambda_1(u))^2-1\Big)\psi^2(s^{-})\widetilde{\mathcal{N}}(\textup{d}s,\textup{d}u)=0,\\\hspace{1.5cm}\underset{t\to\infty}{\lim}~t^{-1}~\int^t_0\int_{\mathcal{U}} \lambda_1(u)S(s^{-})\widetilde{\mathcal{N}}(\textup{d}s,\textup{d}u)=0, \hspace{0.3cm}\underset{t\to\infty}{\lim}~t^{-1}~\int^t_0\int_{\mathcal{U}} \lambda_2(u)I(s^{-})\widetilde{\mathcal{N}}(\textup{d}s,\textup{d}u)=0,\\\hspace{3cm}\mbox{and}\hspace{0.2cm}\underset{t\to\infty}{\lim}~t^{-1}~\int^t_0\int_{\mathcal{U}}\lambda_4(u)\mathcal{D}(s^{-})\widetilde{\mathcal{N}}(\textup{d}s,\textup{d}u)=0\hspace{0.3cm}\mbox{a.s.}$
\end{itemize}
\label{lem1m}
\end{lem}
\begin{rem}
By using the same approach \textcolor{black}{adopted} in Lemma 2.5 of \cite{pr}, we can easily prove the last result. Note that the hypothesis \textbf{($\mathcal{H}_5$)} is an ameliorated version of it corresponding hypothesis frequently used in many previous works\textcolor{black}{, f}or example, \cite{lev1,lev2,lev3}. Therefore, the adoption of $\chi_{1,p}$ in our paper raises the optimality of our calculus and results.
\end{rem}
\begin{rem}
In the absence of Lévy jumps (see for example \cite{sta}), the stationary distribution expression is used to estimate the time averages of the auxiliary process solution by employing the ergodic theorem \cite{mo}. Unluckily, the said expression is still unknown in the case of the Lévy noise. This issue is implicitly mentioned in \cite{58,sharp} as an open question, and the authors presented the threshold analysis of their model with an unknown stationary distribution formula. In this article, we propose an alternative method to establish the exact expression of the threshold parameter without having recourse to the use of ergodic theorem. This new idea that we propose is presented in the following lemma.
\end{rem}
\begin{lem}
Assume \textcolor{black}{that the conditions \textbf{($\mathcal{H}_4$)} and \textbf{($\mathcal{H}_5$)} hold.} Let $\psi(t)$ be the solution of (\ref{s3}) with an initial value $\psi(0)\in\R_{+}$. \textcolor{black}{Then $\dis\chi_2=2\mu_1-\sigma_1^2-\int_{\mathcal{U}}\lambda_1^2(u)\nu(\textup{d}u)>0$,} and  
\begin{itemize}
\item $\dis\underset{t\to\infty}{\lim}t^{-1}~\int^t_0\psi(s)\textup{d}s=\frac{A}{\mu_1}\hspace{0.3cm}\mbox{a.s.}$
\item $\dis\underset{t\to\infty}{\lim}t^{-1}~\int^t_0\psi^2(s)\textup{d}s=\frac{2A^2}{\mu_1\chi_2}\hspace{0.3cm}\mbox{a.s.}$
\end{itemize}
\label{lemmas}
\end{lem}
\begin{proof}
Integrating from $0$ to $t$ on both sides of (\ref{s3}) yields
\begin{align*}
\frac{\psi(t)-\psi(0)}{t}=A-\frac{\mu_1}{t}\int^t_0\psi(s)\textup{d}s+\frac{\sigma_1}{t}\int_0^t \psi(s)\dt \mathcal{W}_1(s)+t^{-1}~\int^t_0\int_Z \lambda_1(u)\psi(s^{-})\widetilde{\mathcal{N}}(\textup{d}s,\textup{d}u).
\end{align*}
Clearly, we can derive that
\begin{align*}
t^{-1}~\int^t_0\psi(s)\textup{d}s=\frac{A}{\mu_1}-\frac{\psi(t)-\psi(0)}{\mu_1 t}+\frac{\sigma_1}{\mu_1 t}\int_0^t \psi(s)\dt \mathcal{W}_1(s)+\frac{1}{\mu_1 t}\int^t_0\int_Z \lambda_1(u)\psi(s^{-})\widetilde{\mathcal{N}}(\textup{d}s,\textup{d}u).
\end{align*}
By Lemma \ref{lem1m}, we get
\begin{align*}
\underset{t\to\infty}{\lim}t^{-1}~\int^t_0\psi(s)\textup{d}s=\frac{A}{\mu_1}\hspace{0.2cm}\mbox{a.s.}
\end{align*}
Now, applying the generalized Itô’s formula to model (\ref{s3}) leads to
\begin{align*}
\dt\psi^2(t)&=\bigg(2\psi(t)\Big(A-\mu_1 \psi(t)\Big)+\sigma_1^2\psi^2(t)+\int_{\mathcal{U}}\psi^2(t)\Big((1+\lambda_1(u))^2-1-2\lambda_1(u)\Big)\nu(\textup{d}u)\bigg)\textup{d}t\\&\;\;\;+2\sigma_1 \psi^2(t)\dt \mathcal{W}_1(t)+\int_{\mathcal{U}} \psi^2(t^{-})\Big((1+\lambda_1(u))^2-1\Big)\widetilde{\mathcal{N}}(\textup{d}t,\textup{d}u).
\end{align*}
Integrating both sides of the last expression from $0$ to $t$ \textcolor{black}{and then dividing by $t$}, yields
\textcolor{black}{
\begin{align*}
\dfrac{\psi^2(t)-\psi^2(0)}{t}&= 2A\times\dfrac{1}{t}\int^t_0\psi(s)\textup{d}s-\overbrace{\bigg(2\mu_1-\sigma_1^2-\int_{\mathcal{U}}\lambda_1^2(u)\nu(\textup{d}u)\bigg)}^{\chi_2}\times\dfrac{1}{t}\int^t_0\psi^2(s)\textup{d}s\\&\;\;\;+2\sigma_1\times\dfrac{1}{t}\int^t_0\psi^2(s)\dt \mathcal{W}_1(s)+\dfrac{1}{t}\int^t_0\int_Z\psi^2(s^-)\Big((1+\lambda_1(u))^2-1\Big)\widetilde{\mathcal{N}}(\textup{d}s,\textup{d}u).
\end{align*}}
 Therefore
\textcolor{black}{\begin{align}\label{ajout}
\dfrac{\chi_2}{t}~\int^t_0\psi^2(s)\textup{d}s &=\frac{2A}{ t}\int^t_0\psi(s)\textup{d}s+\frac{\big(\psi^2(0)-\psi^2(t)\big)}{t}+\frac{2\sigma_1}{t}\int^t_0\psi^2(s)\dt \mathcal{W}_1(s)\\&\;\;\;+\frac{1}{ t}\int^t_0\int_{\mathcal{U}}\psi^2(s^-)\Big((1+\lambda_1(u))^2-1\Big)\widetilde{\mathcal{N}}(\textup{d}s,\textup{d}u).
\end{align}}
\textcolor{black}{Clearly, $\chi_2\neq 0$, because if it is not the case, we will obtain by letting $t$ go to infinity in \eqref{ajout} $\dfrac{2A^2}{\mu}=0$, which is obviously impossible. So, and using Lemma \ref{lem1m}, we can easily verify that}
\begin{align*}
\underset{t\to\infty}{\lim}t^{-1}~\int^t_0\psi^2(s)\textup{d}s=\frac{2A^2}{\mu_1\chi_2}\hspace{0.2cm}\mbox{a.s.}
\end{align*}
\textcolor{black}{and since $\displaystyle{t^{-1}~\int^t_0\psi^2(s)\textup{d}s>0}$ for all $t>0$, we can conclude also that $\frac{2A^2}{\mu_1\chi_2}>0$ and then $\chi_2>0$. Hence the proof is completed. }
\end{proof}
We are now in the position to state and prove the main results of this paper. In the following, we always presume that the hypotheses \textbf{($\mathcal{H}_1$)}-\textbf{($\mathcal{H}_5$)} hold.
\subsection{Stochastic extinction of the epidemic}\label{sec21}
In this subsection, we give a sufficient condition for the stochastic extinction of the disease in the system (\ref{s2}), but before stating
the main result, we shall first recall the concept of the stochastic extinction.
\begin{defn}
\textcolor{black}{Let $\mathcal{X}(t)$ be a stochastic process that describes the  evolution of an infectious disease under a host population.}
\begin{itemize}
\item The \textcolor{black}{disease} is said to be exponentially extinct if $\underset{t\to\infty}{\limsup}~t^{-1}\ln \mathcal{X}(t)<0$ a.s.
\item The \textcolor{black}{ disease} is said to be stochastically extinct, or extinctive, if $\underset{t\to\infty}{\lim}\mathcal{X}(t)=0$ a.s.
\end{itemize}
\end{defn}
\begin{rem}
Obviously, it can be seen from \textcolor{black}{the}  above definitions that the \textcolor{black}{exponential} extinction implies the extinction.
\end{rem}
For brevity and simplicity in writing the next result, we adopt the following notations:
\begin{itemize}
\item $\dis\Upsilon\triangleq\min\{\mu_2+\gamma,\eta\}(\sqrt{\mathcal{T}^{\star}}-1)\mathds{1}_{\{\mathcal{T}^{\star}\leq 1\}}+\max\{\mu_2+\gamma,\eta\}(\sqrt{\mathcal{T}^{\star}}-1)\mathds{1}_{\{\mathcal{T}^{\star}> 1\}}$.
\item $\dis\Lambda\triangleq\sigma_1^2+\int_{\mathcal{U}}\lambda_1^2(u)\nu(\textup{d}u)$\hspace{0.3cm} and \hspace{0.3cm}$\dis\Sigma\triangleq\Big(2\big(\sigma_2^{-2}+\sigma_4^{-2}\big)\Big)^{-1}$.
\item $\dis\bar{\aleph}(u)\triangleq\Big(\ln\big(1+\lambda_2(u)\wedge\lambda_4(u)\big)-\lambda_2(u)\wedge\lambda_4(u)\Big)\times\mathds{1}_{\{\lambda_2(u)\wedge\lambda_4(u)>0\}}$.
\item $\dis\underline{\aleph}(u)\triangleq\Big(\ln\big(1+\lambda_2(u)\vee\lambda_4(u)\big)-\lambda_2(u)\vee\lambda_4(u)\Big)\times\mathds{1}_{\{\lambda_2(u)\vee\lambda_4(u)\leq 0\}}$.
\item $\dis\aleph(u)\triangleq\bar{\aleph}(u)+\underline{\aleph}(u)$\hspace{0.3cm} and\hspace{0.3cm} $\dis\Pi\triangleq\int_{\mathcal{U}}\aleph(u)\nu(\textup{d}u)$.
\item $\dis\Theta\triangleq\Upsilon+\Pi-\Sigma+\eta\Bigg(\frac{\mathcal{T}^{\star}\Lambda}{\chi_2}\Bigg)^{\frac{1}{2}}$.
\item \textcolor{black}{For any vector $v\in\mathbb{R}^n$, we denote its transpose by $v^T$}.
\end{itemize}
\begin{thm}
Let us denote by $(S(t),I(t),\mathcal{D}(t))$ the solution of the stochastic system (\ref{s2}) that starts from a given initial data $(S(0),I(0),\mathcal{D}(0))\in\R^3_+$.  \textcolor{black}{Under the hypotheses \textbf{($\mathcal{H}_1$)}-\textbf{($\mathcal{H}_5$)}, we have}
\begin{align*}
\underset{t\to\infty}{\lim\sup}~t^{-1}\ln\Bigg(\frac{1}{\mu_2+\gamma}I(t)+\frac{\sqrt{\mathcal{T}^{\star}}}{\eta}\mathcal{D}(t)\Bigg)\leq \Theta\hspace*{0.3cm}\mbox{a.s.}
\end{align*}
Notably, if $\Theta<0$, then the epidemic will go to zero exponentially with probability one. Consequently, 
\begin{align*}
\underset{t\to\infty}{\lim} I(t)=0\hspace{0.3cm}\mbox{and}\hspace{0.3cm}\underset{t\to\infty}{\lim} \mathcal{D}(t)=0\hspace*{0.2cm}\mbox{a.s.}
\end{align*}
\label{extinction}
\end{thm}
\begin{proof}
Our proof starts with the use of Theorem 1.4 in \cite{ma} to establish that there is a left eigenvector of the following matrix $$\mathfrak{M}_0=\left(%
\begin{array}{cc}
0& \frac{\beta A}{\mu_1(\mu_2+\gamma)}\\
1&0
\end{array}%
\right)$$
corresponding to $\sqrt{\mathcal{T}^{\star}}$. This vector will be denoted by $(e_1,e_2)=(1,\sqrt{\mathcal{T}^{\star}})$. Then, $\sqrt{\mathcal{T}^{\star}}(e_1,e_2)=(e_1,e_2)\mathfrak{M}_0$. On the other hand, we define a $\mathcal{C}^2$-function $\textcolor{black}{\mathcal{M}}:\R^2_+\textcolor{black}{\rightarrow}\R_+$ by
$$\mathcal{M}(I(t),\mathcal{D}(t))=\omega_1 I(t)+\omega_2 \mathcal{D}(t),$$
where $\omega_1=\frac{e_1}{\mu_2+\gamma}$ and $\omega_2=\frac{e_2}{\eta}$. By applying the generalized Itô's formula with Lévy jumps we obtain
\begin{align*}
\dt\ln \mathcal{M}(I(t),\mathcal{D}(t))&=\mathcal{L}\ln \mathcal{M}(I(t),\mathcal{D}(t))\textup{d}t+\frac{1}{\omega_1 I(t)+\omega_2D(t)}\Big\{\omega_1\sigma_2I(t)\dt \mathcal{W}_2(t)+\omega_2\sigma_4D(t)\dt \mathcal{W}_3(t)\Big\}\\&\;\;\;+\int_{\mathcal{U}}\ln\bigg(1+\frac{\omega_1\lambda_2(u)I(t)+\omega_2\lambda_4(u)\mathcal{D}(t)}{\omega_1 I(t)+\omega_2D(t)}\bigg)\widetilde{\mathcal{N}}(\textup{d}t,\textup{d}u),
\end{align*}
where
\begin{align*}
\mathcal{L}\ln \mathcal{M}(I(t),\mathcal{D}(t))&=\frac{1}{\omega_1 I(t)+\omega_2D(t)}\Big\{\omega_1\Big(\beta S(t)\mathcal{D}(t)-(\mu_2+\gamma)I(t)\Big)+\omega_2\eta\Big(I(t)-\mathcal{D}(t)\Big)\Big\}\\&\;\;\;-\frac{1}{2(\omega_1 I(t)+\omega_2D(t))^2}\Big\{\omega_1^2\sigma_2^2I^2(t)+\omega_2^2\sigma_4^2D^2(t)\Big\}\\&\;\;\;+\int_{\mathcal{U}}\Bigg[\ln\Bigg(1+\frac{\omega_1\lambda_2(u)I(t)+\omega_2\lambda_4(u)\mathcal{D}(t)}{\omega_1 I(t)+\omega_2D(t)}\Bigg)-\frac{\omega_1\lambda_2(u)I(t)+\omega_2\lambda_4(u)\mathcal{D}(t)}{\omega_1 I(t)+\omega_2D(t)}\Bigg]\nu(\textup{d}u).
\end{align*}
Moreover, it is easy to show the following inequality
\begin{align*}
\left(\frac{1}{\sigma_2^2}+\frac{1}{\sigma_4^2}\right)\times\Big(\omega_1^2\sigma_2^2I^2(t)+\omega_2^2\sigma_4^2\mathcal{D}^2(t)\Big)\geq\left(\frac{1}{\sigma_2}\omega_1\sigma_2I(t)+\frac{1}{\sigma_4}\omega_2\sigma_4\mathcal{D}(t)\right)^2.
\end{align*}
In order to find an optimal and good majorization, we adopt the fact that
\begin{align}
\int_{\mathcal{U}}\Bigg[\ln\Bigg(1+\frac{\omega_1\lambda_2(u)I(t)+\omega_2\lambda_4(u)\mathcal{D}(t)}{\omega_1 I(t)+\omega_2D(t)}\Bigg)-\frac{\omega_1\lambda_2(u)I(t)+\omega_2\lambda_4(u)\mathcal{D}(t)}{\omega_1 I(t)+\omega_2D(t)}\Bigg]\nu(\textup{d}u)\leq \Pi.
\label{optim}
\end{align}
By using the last two results, we get
\begin{align*}
\mathcal{L}\ln \mathcal{M}(I(t),\mathcal{D}(t))&\leq \frac{\omega_1\beta \mathcal{D}(t)}{\omega_1 I(t)+\omega_2D(t)}\Big(S(t)-\frac{A}{\mu_1}\Big)+\Pi-\Sigma\\&\;\;\;+\frac{1}{\omega_1 I(t)+\omega_2D(t)}\Bigg\{\omega_1\Big(\frac{\beta A}{\mu_1}\mathcal{D}(t)-(\mu_2+\gamma)I(t)\Big)+\omega_2\eta\Big(I(t)-\mathcal{D}(t)\Big)\Bigg\}.
\end{align*}
By the stochastic comparison theorem, we have
\begin{align*}
\mathcal{L}\ln \mathcal{M}(I(t),\mathcal{D}(t))&\leq \frac{\omega_1\beta \mathcal{D}(t)}{\omega_1 I(t)+\omega_2D(t)}\Big(\psi(t)-\frac{A}{\mu_1}\Big)+\Pi-\Sigma\\&\;\;\;+\frac{1}{\omega_1 I(t)+\omega_2D(t)}\Bigg\{\frac{e_1}{\mu_2+\gamma}\Bigg(\frac{\beta A}{\mu_1}\mathcal{D}(t)-(\mu_2+\gamma)I(t)\Bigg)+\frac{e_2}{\eta}\Big(\eta I(t)-\eta \mathcal{D}(t)\Big)\Bigg\}.
\end{align*}
Then, we obtain that
\begin{align*}
\mathcal{L}\ln \mathcal{M}(I(t),\mathcal{D}(t))&\leq \frac{\omega_1\beta }{\omega_2}\Big |\psi(t)-\frac{A}{\mu_1}\Big |+\Pi-\Sigma+\frac{1}{\omega_1 I(t)+\omega_2D(t)}(e_1,e_2)\Big(\textcolor{black}{\mathfrak{M}_0}(I(t),\mathcal{D}(t))^T-(I(t),\mathcal{D}(t))^T\Big)\\
&=\frac{\omega_1\beta }{\omega_2}\Big |\psi(t)-\frac{A}{\mu_1}\Big |+\Pi-\Sigma+\frac{1}{\omega_1 I(t)+\omega_2D(t)}\big(\sqrt{\mathcal{T}^{\star}}-1\big)\big(e_1I(t)+e_2D(t)\big)\\
&=\frac{\omega_1\beta }{\omega_2}\Big |\psi(t)-\frac{A}{\mu_1}\Big |+\Pi-\Sigma+\frac{1}{\omega_1 I(t)+\omega_2D(t)}\big(\sqrt{\mathcal{T}^{\star}}-1\big)\big(\omega_1(\mu_2+\gamma)I(t)+\eta\omega_2D(t)\big)\\&\leq \Upsilon+\Pi-\Sigma+\frac{\omega_1\beta }{\omega_2}\Big |\psi(t)-\frac{A}{\mu_1}\Big |.
\end{align*}
Hence, we deduce that
\begin{align*}
\dt\ln \mathcal{M}(I(t),\mathcal{D}(t))&\leq \big(\Upsilon+\Pi-\Sigma\big)\dt t+\frac{\omega_1\beta }{\omega_2}\Big |\psi(t)-\frac{A}{\mu_1}\Big |\textup{d}t\\&\;\;\;+\frac{1}{\omega_1 I(t)+\omega_2D(t)}\Big\{\omega_1\sigma_2I(t)\dt \mathcal{W}_2(t)+\omega_2\sigma_4D(t)\dt \mathcal{W}_4(t)\Big\}\\&\;\;\;+\int_{\mathcal{U}}\ln\big(1+\lambda(u)\big)\widetilde{\mathcal{N}}(\textup{d}t,\textup{d}u),
\end{align*}
where $\lambda(u)=\max\{\lambda_2(u),\lambda_4(u)\}$. Now, by integrating both sides of the last inequality and dividing by $t$, we find immediately that
\begin{align}
t^{-1}~\ln \mathcal{M}(I(t),\mathcal{D}(t))&\leq t^{-1}~\ln \mathcal{M}(I(0),\mathcal{D}(0))+\Upsilon+\Pi-\Sigma\nonumber\\&\;\;\;+\frac{\omega_1\beta }{\omega_2 \textcolor{black}{t} }\int^t_0\Big |\psi(s)-\frac{A}{\mu_1}\Big |\textup{d}s+t^{-1}~\mathcal{J}_3(t)+t^{-1}~\mathcal{J}_4(t),
\label{mint}
\end{align}
where
\begin{align*}
\color{black}{\mathcal{J}_3(t)}&\color{black}{=\int^t_0 \frac{\sigma_2\omega_1I(s)}{\omega_1 I(s)+\omega_2D(s)} \dt\mathcal{W}_2(s)+\int^t_0\frac{\sigma_4\omega_2D(s)}{\omega_1 I(s)+\omega_2D(s)} \dt \mathcal{W}_4(s),}\\
\mathcal{J}_4(t)&=\int^t_0\int_{\mathcal{U}}\ln\big(1+\lambda(u)\big)\widetilde{\mathcal{N}}(\textup{d}s,\textup{d}u).
\end{align*}
\textcolor{black}{It is easy to check that $\mathcal{J}_3(t)$ is a local martingale with finite quadratic variation, and by the hypothesis \textbf{($\mathcal{H}_3$)} we can affirm that $\mathcal{J}_4(t)$ is also a local martingale with finite quadratic variation. By the strong law of large numbers for local martingales \cite{mo}, we get}
\begin{align*}
\underset{t\to\infty}{\lim}~t^{-1}~\mathcal{J}_3(t)=0 \;\;\;\mbox{a.s}\;\;\;\;\mbox{and}\;\;\;\;\underset{t\to\infty}{\lim}~t^{-1}~\mathcal{J}_4(t)=0 \;\;\;\;\mbox{a.s\textcolor{black}{.}}
\end{align*}
Now, by using the Hölder's inequality, we deduce that
\begin{align*}
t^{-1}~\int^t_0\Big |\psi(s)-\frac{A}{\mu_1}\Big |\textup{d}s\leq ~t^{-\frac{1}{2}}\Bigg(\int^t_0\Big (\psi(s)-\frac{A}{\mu_1}\Big )^2\textcolor{black}{\textup{d}}s\Bigg)^{\frac{1}{2}}=\Bigg(t^{-1}~\int^t_0\Big(\psi^2(s)-\frac{2A}{\mu_1}\psi(s)+\Big(\frac{A}{\mu_1}\Big)^2\Big)\textup{d}s\Bigg)^{\frac{1}{2}}.
\end{align*}
It follows from Lemme \ref{lemmas} that
\begin{align*}
\underset{t\to\infty}{\lim}t^{-1}~\int^t_0\Big |\psi(s)-\frac{A}{\mu_1}\Big |\textup{d}s&\leq\Bigg(\frac{2A^2}{\mu_1\chi_2}-\frac{2A^2}{\mu_1^2}+\frac{A^2}{\mu_1^2}\Bigg)^{\frac{1}{2}}=\left(\frac{A^2\Big(\sigma_1^2+\int_{\mathcal{U}}\lambda_1^2(u)\nu(\textup{d}u)\Big)}{\mu_1^2\chi_2}\right)^{\frac{1}{2}}.
\end{align*}
Taking the superior limit on both sides of \eqref{mint} leads to
\begin{align*}
\underset{t\to\infty}{\lim\sup}~t^{-1}~\ln \mathcal{ M}(I(t),\mathcal{D}(t))&\leq \Upsilon+\Pi-\Sigma+\frac{\omega_1\beta }{\omega_2 }\left(\frac{A^2\Big(\sigma_1^2+\int_{\mathcal{U}}\lambda_1^2(u)\nu(\textup{d}u)\Big)}{\mu_1^2\chi_2}\right)^{\frac{1}{2}}\hspace*{0.3cm}\mbox{a.s.}
\end{align*}
Which implies,
\begin{align*}
\underset{t\to\infty}{\lim\sup}~t^{-1}~\ln \mathcal{M}(I(t),\mathcal{D}(t))&\leq  \Upsilon+\Pi-\Sigma+\eta\Bigg(\frac{\mathcal{T}^{\star}\Lambda}{\chi_2}\Bigg)^{\frac{1}{2}}=\Theta\hspace*{0.3cm}\mbox{a.s.}
\end{align*}
That is to say, if $\Theta<0$, then $\underset{t\to\infty}{\lim\sup}~t^{-1}~\ln I(t)<0$, and $\underset{t\to\infty}{\lim\sup}~t^{-1}~\ln \mathcal{D}(t)<0$ a.s., which implies in turn that the disease will die out with probability one and this completes the proof.
\end{proof}
\begin{rem}
When the jump\textcolor{black}{s} coefficients $\lambda_i(u)$ ($i=1,2,4$) and the white intensities $\sigma_i$ ($i=2,4$) degenerate to zero, our results in Theorem \ref{extinction} coincide with Theorem 2.3 in \cite{da}.  Therefore, our results generalize the consequence of the mentioned paper.
\end{rem}
\subsection{Persistence in mean of the epidemic}\label{sec3}
The study of the persistence in the mean is a significant characteristic to know more about epidemic dynamics. For this reason, in this section, we will give the condition for the disease persistence, but before stating the main result, we shall first recall the concept of persistence in the mean.
\begin{defn}
An infected population $\mathcal{I}(t)$ is said to be be strongly persistent in the mean, or just persistent in the mean, if $\dis\underset{t\to\infty}{\lim\inf} ~t^{-1}~\int^t_0\mathcal{I}(s)\dt s\textcolor{black}{>0}$ almost surely.
\end{defn}
For simplicity of notation, we define the following quantity
\begin{align*}
\widetilde{\mathcal{T}^{\star}}=\beta\Bigg(\frac{A}{\mu_1+\bar{\sigma}_1}\Bigg)\Bigg((\mu_2+\gamma+\bar{\sigma}_2)+\beta \left(\frac{A}{\mu_1+\bar{\sigma}_1}\right)\frac{\bar{\sigma}_4}{\eta}\Bigg)^{-1},
\end{align*}
where $\dis\bar{\sigma}_i\triangleq0.5\sigma_i^2+\int_{\mathcal{U}}\Big(\lambda_i(u)-\ln(1+\lambda_i(u)\Big)\nu(\textup{d}u)$, $i=1,2,4$.
\begin{thm}
Let $(S(t),I(t),\mathcal{D}(t))$ be the solution of (\ref{s2}) with any initial data $(S(0),I(0),\mathcal{D}(0))\in\R^3_+$. The stochastic model (\ref{s2}) has the following property: if $\widetilde{\mathcal{T}^{\star}}>1$ holds, then the disease $I(t)$ persists in the mean almost surely.
\label{persistence}
\end{thm}
\begin{proof}
Begin by considering the following function
\begin{align*}
\mathcal{Z}(S(t),I(t),\mathcal{D}(t))=-\mathfrak{c}_1\ln S(t)-\ln I(t)-\mathfrak{c}_2 \ln \mathcal{D}(t)+\mathfrak{c}_3 \mathcal{D}(t).
\end{align*}
where $\mathfrak{c}_i$, $(i=1,2,3)$ are positive constants to be determined in the following. From Itô’s formula and system \eqref{s2}, we have
\begin{align*}
\dt\mathcal{Z}(S(t),I(t),\mathcal{D}(t))&=\mathcal{L}\mathcal{Z}(S(t),I(t),\mathcal{D}(t))\textup{d}t-\mathfrak{c}_1\sigma_1 \dt \mathcal{W}_1(t)-\sigma_2\dt \mathcal{W}_2(t)-\mathfrak{c}_2\sigma_4 \dt \mathcal{W}_4(t)\\&\;\;\;+\mathfrak{c}_3\sigma_4D(t)\dt \mathcal{W}_4(t)-\int_{\mathcal{U}}\mathfrak{c}_1\ln\big(1+\lambda_1(u)\big)\widetilde{\mathcal{N}}(\textup{d}t,\textup{d}u)-\int_{\mathcal{U}}\ln\big(1+\lambda_2(u)\big)\widetilde{\mathcal{N}}(\textup{d}t,\textup{d}u)\\&\;\;\;-\int_{\mathcal{U}}\mathfrak{c}_2\ln\big(1+\lambda_4(u)\big)\widetilde{\mathcal{N}}(\textup{d}t,\textup{d}u)+\int_{\mathcal{U}}\mathfrak{c}_3\lambda_4(u)\mathcal{D}(t^-)\widetilde{\mathcal{N}}(\textup{d}t,\textup{d}u),
\end{align*}
where
\begin{align*}
\mathcal{L}\textcolor{black}{\mathcal{Z}}(S(t),I(t),\mathcal{D}(t))&=-\frac{\mathfrak{c}_1}{S(t)}(A-\mu_1 S(t)-\beta S(t) \mathcal{D}(t))+\frac{\mathfrak{c}_1\sigma_1^2}{2}-\frac{1}{I(t)}\big(\beta S(t)\mathcal{D}(t)-(\mu_2+\gamma)I(t)\big)\\&\;\;\;+\frac{\sigma_2^2}{2}-\frac{\mathfrak{c}_2\eta}{\mathcal{D}(t)}(I(t)-\mathcal{D}(t))+\frac{\mathfrak{c}_2\sigma_4^2}{2}+\mathfrak{c}_3\eta\big(I(t)-\mathcal{D}(t)\big)\\&\;\;\;+\int_{\mathcal{U}}\mathfrak{c}_1\big(\lambda_1(u)-\ln(1+\lambda_1(u)\big)\nu(\textup{d}u)+\int_{\mathcal{U}}\big(\lambda_2(u)-\ln(1+\lambda_2(u)\big) \nu(\textup{d}u)\\&\;\;\;+\int_{\mathcal{U}}\mathfrak{c}_2\big(\lambda_4(u)-\ln(1+\lambda_4(u)\big)\nu(\textup{d}u).
\end{align*}
We then find that
\begin{align*}
\mathcal{L}\mathcal{Z}(S(t),I(t),\mathcal{D}(t))&=-\frac{\beta S(t)\mathcal{D}(t)}{I(t)}-\frac{\mathfrak{c}_1A}{S(t)}-\frac{\mathfrak{c}_2\eta I(t)}{\mathcal{D}(t)}+(\mathfrak{c}_1\beta-\mathfrak{c}_3\eta)\mathcal{D}(t)\\&\;\;\;+\mathfrak{c}_1(\mu_1+\bar{\sigma}_1)+\mathfrak{c}_2(\eta+\bar{\sigma}_4)+(\mu_2+\gamma+\bar{\sigma}_2)+\mathfrak{c}_3\eta I(t)\\&\leq -3\big(\beta A \eta \mathfrak{c}_1 \mathfrak{c}_2\big)^{\frac{1}{3}}+(\mathfrak{c}_1\beta-\mathfrak{c}_3\eta)\mathcal{D}(t)+\mathfrak{c}_1(\mu_1+\bar{\sigma}_1)\\&\;\;\;+\mathfrak{c}_2(\eta+\bar{\sigma}_4)+(\mu_2+\gamma+\bar{\sigma}_2)+\mathfrak{c}_3\eta I(t).
\end{align*}
By choosing 
\begin{align*}
\mathfrak{c}_1&=\beta\Big(\frac{A}{\mu_1+\bar{\sigma}_1}\Big)^2(\eta+\bar{\sigma}_4)/A\eta,\\
\mathfrak{c}_2&=\beta\Big(\frac{A}{\mu_1+\bar{\sigma}_1}\Big)/(\eta+\bar{\sigma}_4),\\
\mathfrak{c}_3&=\mathfrak{c}_1\beta/\eta,
\end{align*}
we may actually obtain that
\begin{align*}
\mathcal{L}\mathcal{Z}(S(t),I(t),\mathcal{D}(t))&\leq -\beta\Big(\frac{A}{\mu_1+\bar{\sigma}_1}\Big)+(\mu_2+\gamma+\bar{\sigma}_2)+\beta \Big(\frac{A}{\mu_1+\bar{\sigma}_1}\Big)\frac{\bar{\sigma}_4}{\eta}+\mathfrak{c}_1\beta I(t)\\&=-\beta\Big(\frac{A}{\mu_1+\bar{\sigma}_1}\Big)\left(1-\frac{1}{\widetilde{\mathcal{T}^{\star}}}\right)+\mathfrak{c}_1\beta I(t).
\end{align*}
Hence, we get
\begin{align*}
\dt \mathcal{Z}(S(t),I(t),\mathcal{D}(t))&\leq \Bigg(-\beta\Big(\frac{A}{\mu_1+\bar{\sigma}_1}\Big)\left(1-\frac{1}{\widetilde{\mathcal{T}^{\star}}}\right)+\mathfrak{c}_1\beta I(t)\Bigg) \textup{d}t-\mathfrak{c}_1\sigma_1 \dt \mathcal{W}_1(t)-\sigma_2\dt \mathcal{W}_2(t)\\&\;\;\;-\mathfrak{c}_2\sigma_4 \dt \mathcal{W}_4(t)+\mathfrak{c}_3\sigma_4D(t)\dt \mathcal{W}_4(t)-\int_{\mathcal{U}}\mathfrak{c}_1\ln\big(1+\lambda_1(u)\big)\widetilde{\mathcal{N}}(\textup{d}t,\textup{d}u)\\&\;\;\;-\int_{\mathcal{U}}\ln\big(1+\lambda_2(u)\big)\widetilde{\mathcal{N}}(\textup{d}t,\textup{d}u)-\int_{\mathcal{U}}\mathfrak{c}_2\ln\big(1+\lambda_4(u)\big)\widetilde{\mathcal{N}}(\textup{d}t,\textup{d}u)\\&\;\;\;+\int_{\mathcal{U}}\mathfrak{c}_3\lambda_4(u)\mathcal{D}(t^-)\widetilde{\mathcal{N}}(\textup{d}t,\textup{d}u).
\end{align*}
Integrating from $0$ to $t$ and dividing by $t$ on both sides of the last inequality, yields
\begin{align*}
~t^{-1}~\mathcal{Z}(S(t),I(t),\mathcal{D}(t))&\leq ~t^{-1}~\mathcal{Z}(S(0),I(0),\mathcal{D}(0))-\beta\Big(\frac{A}{\mu_1+\bar{\sigma}_1}\Big)\left(1-\frac{1}{\widetilde{\mathcal{T}^{\star}}}\right)\\&\;\;\;+\mathfrak{c}_1\beta~t^{-1}~\int^t_0 I(s) \textup{d}s+~t^{-1}~\mathcal{J}_5(t)+~t^{-1}~\mathcal{J}_6(t),
\end{align*}
where
\begin{align*}
\mathcal{J}_5(t)&=-\Big(\mathfrak{c}_1\sigma_1 \mathcal{W}_1(t)+\sigma_2 \mathcal{W}_2(t)+\mathfrak{c}_2\sigma_4 \mathcal{W}_4(t)\Big)+\mathfrak{c}_3\sigma_4\int^t_0D(s)\dt \mathcal{W}_4(s),\\
\mathcal{J}_6(t)&=-\textcolor{black}{\int_0^t}\int_{\mathcal{U}}\mathfrak{c}_1\ln\big(1+\lambda_1(u)\big)\widetilde{\mathcal{N}}(\textup{d}t,\textup{d}u)-\textcolor{black}{\int_0^t}\int_{\mathcal{U}}\ln\big(1+\lambda_2(u)\big)\widetilde{\mathcal{N}}(\textup{d}t,\textup{d}u)\\&\;\;\;-\textcolor{black}{\int_0^t}\int_{\mathcal{U}}\mathfrak{c}_2\ln\big(1+\lambda_4(u)\big)\widetilde{\mathcal{N}}(\textup{d}t,\textup{d}u)+\textcolor{black}{\int_0^t}\int_{\mathcal{U}}\mathfrak{c}_3\lambda_4(u)\mathcal{D}(t^-)\widetilde{\mathcal{N}}(\textup{d}t,\textup{d}u).
\end{align*}
By using, the strong law of large numbers for local martingales and Lemma \ref{lem1m}, we can obtain 
\begin{align*}
\underset{t\to\infty}{\lim}~t^{-1}~\mathcal{J}_5(t)=0 \;\;\;\mbox{a.s}\;\;\;\;\mbox{and}\;\;\;\;\underset{t\to\infty}{\lim}~t^{-1}~\mathcal{J}_6(t)=0 \;\;\;\;\mbox{a.s}
\end{align*}
 Therefore
\begin{align*}
\underset{t\to\infty}{\lim\inf}~t^{-1}~\int^t_0 I(s) \textup{d}s\geq \frac{1}{\mathfrak{c}_1}\Big(\frac{A}{\mu_1+\bar{\sigma}_1}\Big)\left(1-\frac{1}{\widetilde{\mathcal{T}^{\star}}}\right)>0\hspace*{0.5cm}\mbox{a.s.}
\end{align*}
This shows that the disease \textcolor{black}{persists} in the mean \textcolor{black}{as claimed}.
\end{proof}
\begin{rem}
Persistence in the mean is an important concept in mathematical epidemiology. It captures the long-term survival of the disease even when the population size is quite low at \textcolor{black}{$t=0$}. Moreover, the persistence of the model refers to a situation where the disease is endemic in a population. 
\end{rem}
\section{Numerical simulations}\label{sec4}
This section is devoted to illustrate our theoretical results by employing numerical simulations. In the three following examples, we apply the algorithm presented in \cite{simul} to discretize the disturbed system (\ref{s2}). Using the software \textbf{M}atlab2015b and the parameter values listed in Table \ref{value1}, we numerically simulate the solution of the system (\ref{s2}) with the initial value $ (S (0), I (0), \mathcal{D} (0)) = (0.6 \mathbf{,}\; 0.3 \mathbf{,}\; 0.05) $.  
\begin{center}
\begin{tabular}{||cl||c|c|c|c||}
\hline
 \cellcolor{gray!25} Parameters &\hspace{1.7cm}\cellcolor{gray!25}Description  &\multicolumn{4}{|c||}{\cellcolor{gray!40}Numerical values} \\
\hline \hline
$A$ & Recruitment rate & 0.9  &0.3&0.6&0.6\\
\hline
$\mu_1$ & Natural mortality rate of $S$ & 0.3&0.3&0.4&0.4\\
\hline
$\beta$ & Transmission rate  &0.07 & 1.3&0.35&0.8\\
\hline
$\gamma$ & Recovered rate &0.05& 0.05&0.2&0.3\\
\hline
$\mu_2$ &General mortality of $I$ &0.5 &0.5&0.3&0.3\\
\hline
$\eta$ &Exponentially fading memory rate &0.09 &0.09&0.7&0.2\\
\hline
$\sigma_1$ &Intensity of $\mathcal{W}_1(t)$ &0.15 &0.15&\textcolor{black}{0.2}&\textcolor{black}{0.169}\\
\hline
$\sigma_2$ &Intensity of $\mathcal{W}_2(t)$ &0.25 &0.25&\textcolor{black}{0.15}&\textcolor{black}{0.15}\\
\hline
$\sigma_4$ &Intensity of $\mathcal{W}_4(t)$ &0.27 &0.27&\textcolor{black}{0.13}&\textcolor{black}{0.13}\\
\hline
$\lambda_1$ &Jump intensity of $S$ &0.2 &0.2&\textcolor{black}{0.5}&\textcolor{black}{0.5}\\
\hline
$\lambda_2$ &Jump intensity of $I$ &0.23 &0.23&\textcolor{black}{0.3}&\textcolor{black}{0.3}\\
\hline
$\lambda_4$ &Jump intensity of $\mathcal{D}$ &0.1 &0.1&\textcolor{black}{0.7}&\textcolor{black}{0.7}\\
\hline
\multicolumn{2}{c|}{} &\cellcolor{gray!40} Figure \ref{fig1} &\cellcolor{gray!40} Figure \ref{fig2}&\cellcolor{gray!40} Figure \ref{fig3}&\cellcolor{gray!40} Figure \ref{fig4}\\ 
\cline{3-6}
\end{tabular}
\captionof{table}{\textcolor{black}{Nominal values of the system  parameters and disturbances intensities adopted in the different  simulation examples .}}
\label{value1}
\end{center}
\subsection{The stochastic extinction case}
In order to exhibit the strong random fluctuations effect on epidemic dynamics, we present in Figure \ref{fig1}, the trajectories of the stochastic solution $(S(t),I(t),\mathcal{D}(t))$. We assume that $\mathcal{U}=(0,\infty)$ and $\nu(\mathcal{U})=1$, then by using the parameters listed in Table \ref{value1}, we must check the existence of $p$ such that $\chi_{1,p}>0$. By simple calculation, we easily get $\chi_{1,p}=0.0206$ for $p=2.1$\textbf{.} Then, the condition \textbf{($\mathcal{H}_5$)} is satisfied. With the chosen parameters, we can obtain the following values:
\begin{center}
\begin{tabular}{||c||c||}
\hline
 \cellcolor{gray!25} Expression &\cellcolor{gray!25} Value \\
\hline \hline
$\mathcal{T}^{\star}\triangleq\beta A(\mu_1(\mu_2+\gamma))^{-1}$ & 0.3818\\
\hline
$\Upsilon\triangleq\min\{\mu_2+\gamma,\eta\}(\sqrt{\mathcal{T}^{\star}}-1)$ & -0.0344 \\
\hline
$\Pi\triangleq\int_{\mathcal{U}}\aleph(u)\nu(\textup{d}u)$& -0.0047\\
\hline
$\Sigma\triangleq\Big(2\big(\sigma_2^{-2}+\sigma_4^{-2}\big)\Big)^{-1}$&0.0168\\
\hline
$\Lambda\triangleq\sigma_1^2+\int_{\mathcal{U}}\lambda_1^2(u)\nu(\textup{d}u)$&  0.0625\\
\hline
$\chi_2=2\mu_1-\sigma_1^2-\int_{\mathcal{U}}\lambda_1^2(u)\nu(\textup{d}u)$&  0.5375\\
\hline
$\Theta\triangleq\Upsilon+\Pi-\Sigma+\eta\textcolor{black}{\sqrt{\mathcal{T^\star}}}\Lambda^{\frac{1}{2}}\chi_2^{-\frac{1}{2}}$&-0.0369\\
\hline
\end{tabular}
\captionof{table}{Some expressions and their corresponding values.}
\label{value12}
\end{center}
\textcolor{black}{From Table \ref{value12}, we have $\Theta<0$}, then the condition of Theorem \ref{extinction} is verified. That is to say that  the epidemic dies out exponentially almost surely \textcolor{black}{which is exactly illustrated in Figure \ref{fig1}}.
\begin{figure}[!htb]\centering
\begin{center}$
\begin{array}{cc}
\includegraphics[width=3.4in]{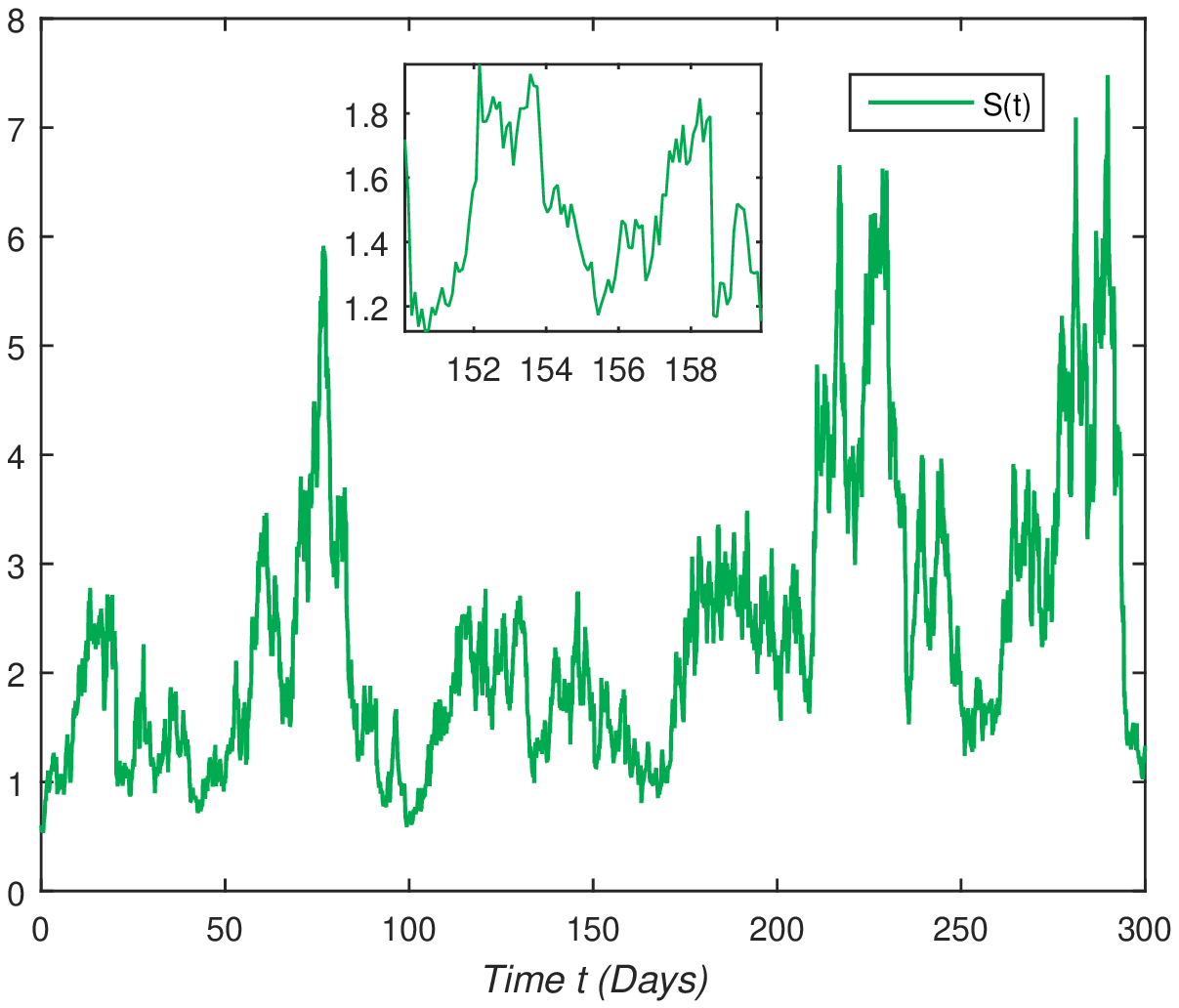} &
\includegraphics[width=3.4in]{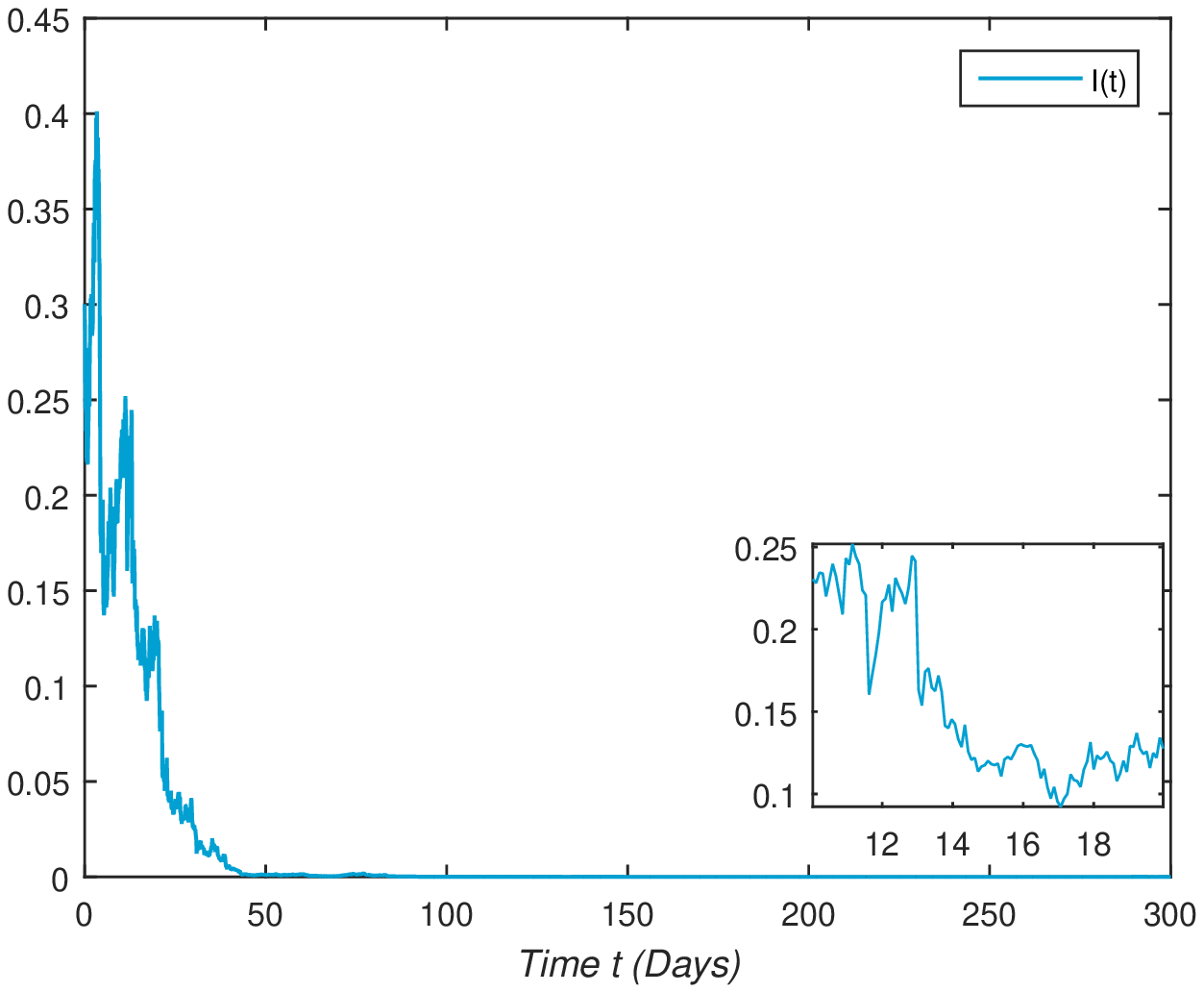}
\end{array}$
$\begin{array}{c}
\includegraphics[width=3.4in]{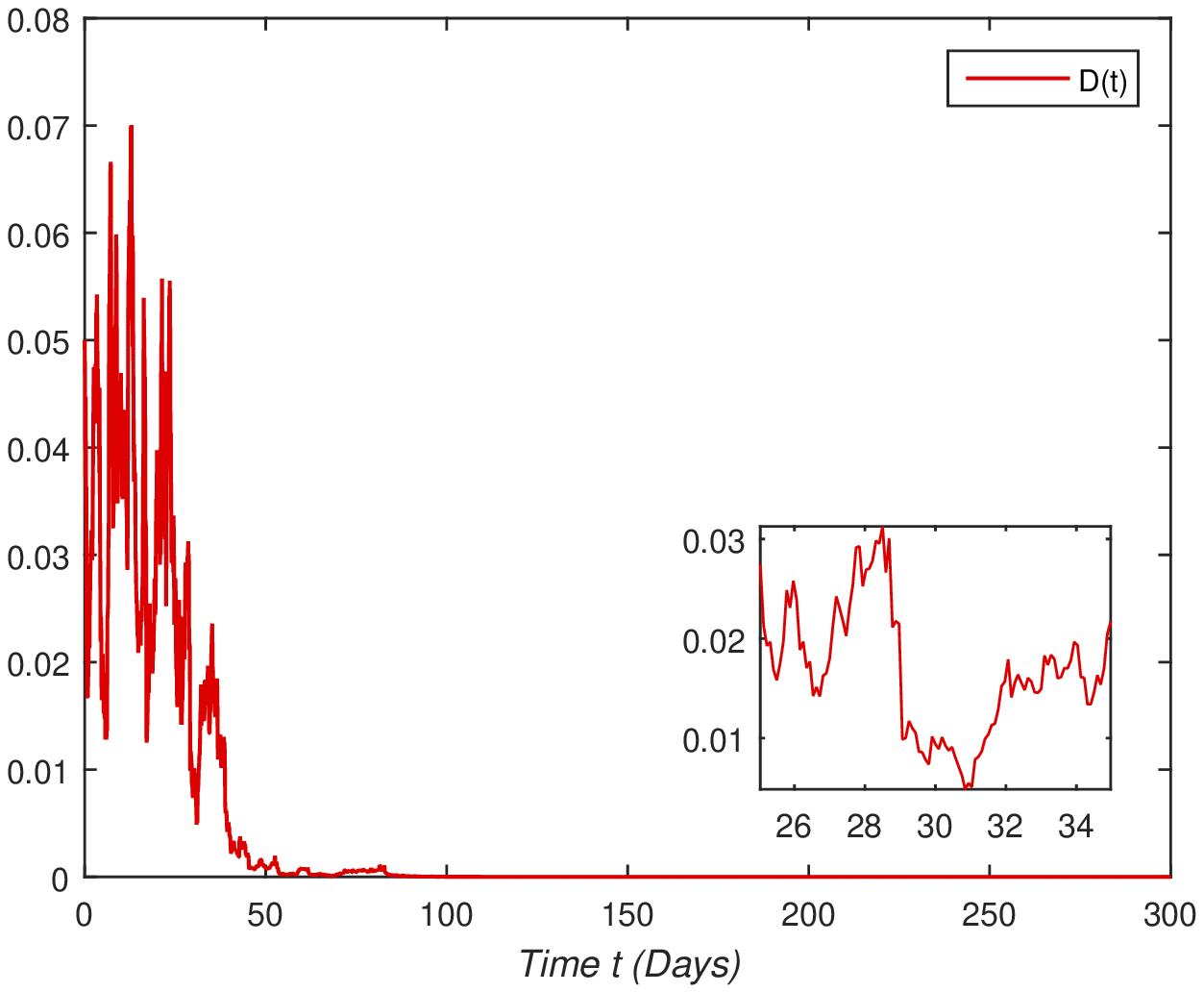}
\end{array}$
\end{center}
\caption{The paths of $S(t)$, $I(t)$ and $\mathcal{D}(t)$ for the stochastic model (\ref{s2}) when $\Theta=-0.0369<0$.}
\label{fig1}
\end{figure}
\subsection{The stochastic persistence case}
Consider the system (\ref{s2}) with parameters \textcolor{black}{appearing} in Table \ref{value1}. Then, we obtain the following values:
\begin{center}
\begin{tabular}{||c||c||}
\hline
 \cellcolor{gray!25} Expression &\cellcolor{gray!25} Value \\
\hline \hline
$\bar{\sigma}_1\triangleq 0.5\sigma_1^2+\int_{\mathcal{U}}\Big(\lambda_1(u)-\ln(1+\lambda_1(u)\Big)\nu(\textup{d}u)$ & 0.0289\\
\hline
$\bar{\sigma}_2\triangleq 0.5\sigma_2^2+\int_{\mathcal{U}}\Big(\lambda_2(u)-\ln(1+\lambda_2(u)\Big)\nu(\textup{d}u)$ &0.0542\\
\hline
$\bar{\sigma}_4\triangleq 0.5\sigma_4^2+\int_{\mathcal{U}}\Big(\lambda_4(u)-\ln(1+\lambda_4(u)\Big)\nu(\textup{d}u)$ & 0.0411\\
\hline
$\widetilde{\mathcal{T}^{\star}}\triangleq \beta\Bigg(\frac{A}{\mu_1+\bar{\sigma}_1}\Bigg)\Bigg((\mu_2+\gamma+\bar{\sigma}_2)+\beta \left(\frac{A}{\mu_1+\bar{\sigma}_1}\right)\frac{\bar{\sigma}_4}{\eta}\Bigg)^{-1}$&1.0344\\
\hline
\end{tabular}
\captionof{table}{Some expressions and their corresponding values.}
\label{value13}
\end{center}
Therefore, $\widetilde{\mathcal{T}^{\star}}>1$. From Figure \ref{fig2}, we observe the persistence of the epidemic $I(t)$ in this case, which agree well with Theorem \ref{persistence}. Furthermore, the solutions $S(t)$ and $D(t)$ are persistent which implies the non-extinction of the stochastic model \eqref{s2}.
\begin{figure}[!htb]\centering
\begin{center}$
\begin{array}{cc}
\includegraphics[width=3.4in]{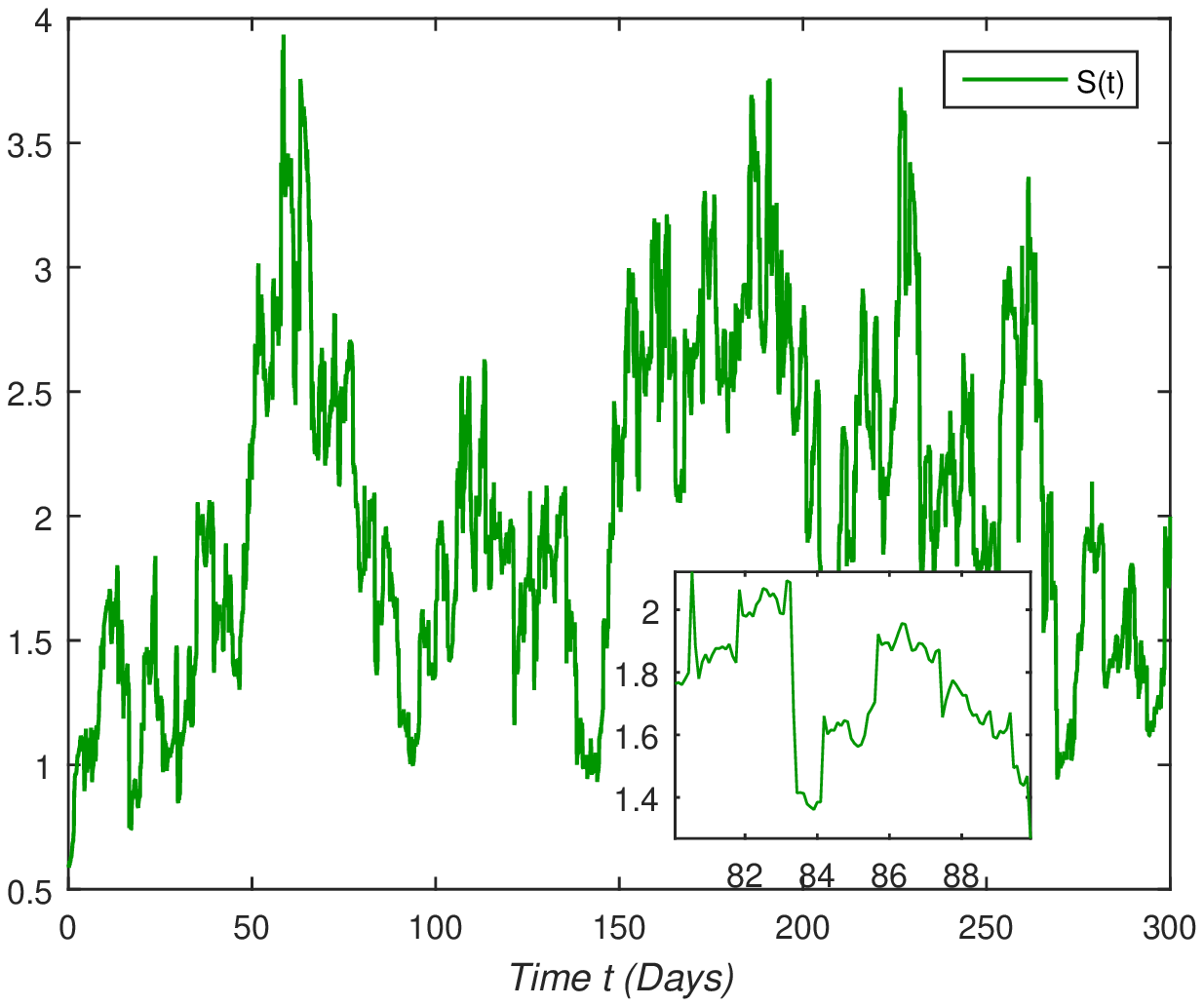} &
\includegraphics[width=3.4in]{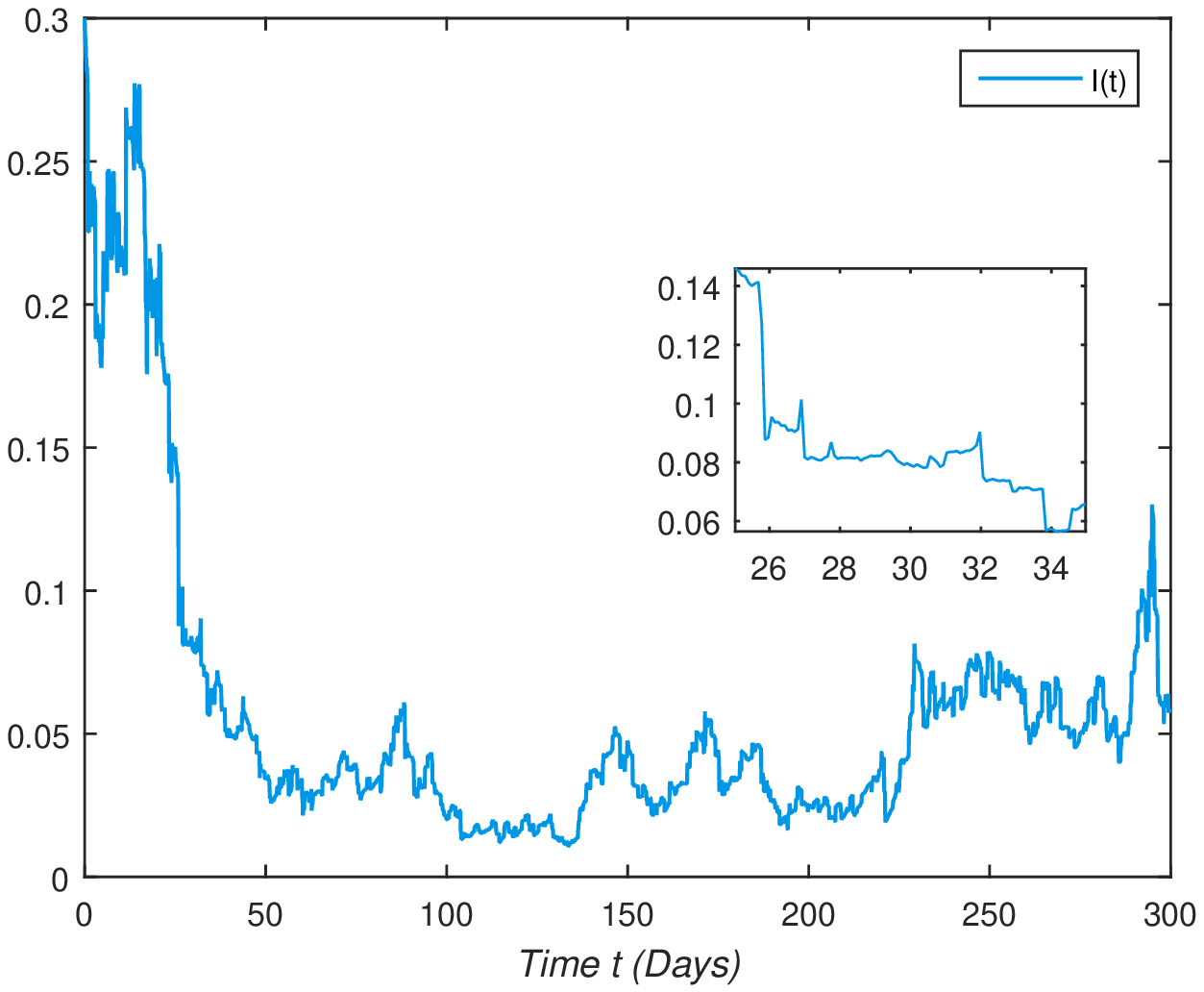}
\end{array}$
$\begin{array}{c}
\includegraphics[width=3.4in]{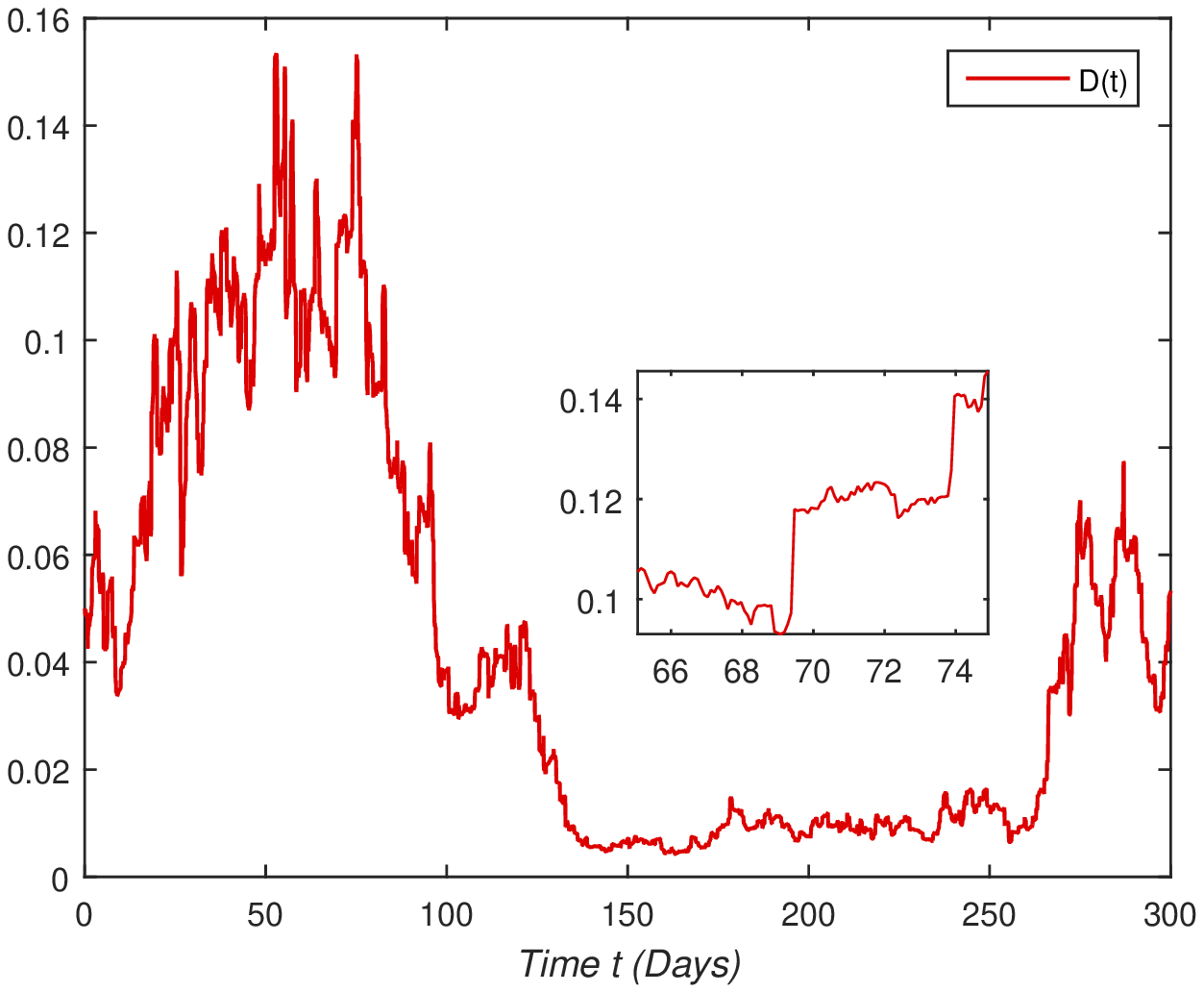}
\end{array}$
\end{center}
\caption{The paths of $S(t)$, $I(t)$ and $\mathcal{D}(t)$ for the stochastic model (\ref{s2}) when $\widetilde{\mathcal{T}^{\star}}=1.0344>1$.}
\label{fig2}
\end{figure}
\subsection{The Lévy jumps effect on the epidemic dynamics}
To find out the effect of white noise and jumps  intensities on epidemic dynamics, in this example, we will compare the trajectories of the following systems:
\begin{itemize}
\item [$\bullet$] The deterministic model \eqref{s12} ($\sigma_i=0$ and $\lambda_i=0$, $i=1,2,4$).
\item [$\bullet$] The stochastic version of \eqref{s12} with degenerate diffusion \cite{da} ($\sigma_2=\sigma_4=0$ and $\lambda_i=0$, $i=1,2,4$).
\item [$\bullet$] The SDE-J system \eqref{s2} ($\sigma_i\neq 0$ and $\lambda_i\neq 0$, $i=1,2,4$).
\end{itemize}
We take the values appearing in Table \ref{value1} which are the same as those used in \cite{da}. For the rest of parameters, we choose $\sigma_1=0.2$, $\sigma_2=0.15$, $\sigma_4=0.13$, $\lambda_1=0.5$, $\lambda_2=0.3$, and $\lambda_4=0.7$. For the sake of a comparison, we choose the following initial value used in $ (S (0), I (0), \mathcal{D} (0)) = (0.2 \mathbf{,}\; 0.3 \mathbf{,}\; 0.4) $ used in \cite{da}. We see from Figure $\ref{fig3}$ that the effects of Lévy jump\textcolor{black}{s} lead to the extinction of the disease while the deterministic model \eqref{s12} and the perturbed model driven by degenerate diffusion both predict persistence. Therefore, we say that the jumps have negative effects on the prevalence of epidemics. This means that jumps can change the asymptotic behavior of the epidemic model significantly. To examine the effect of \textcolor{black}{jumps intensities} on dynamical system \eqref{s12} in the case of persistence,  we shall decrease the intensity $\sigma_1$ to $0.169$ and take other parameter \textcolor{black}{as in the last column of Table} \ref{value1}. From Figure \ref{fig4}, we observe the persistence of the epidemic in all cases with a greater variation in the case of Lévy jumps.
\begin{figure}[!htb]\centering
\begin{center}$
\begin{array}{cc}
\includegraphics[width=3.4in]{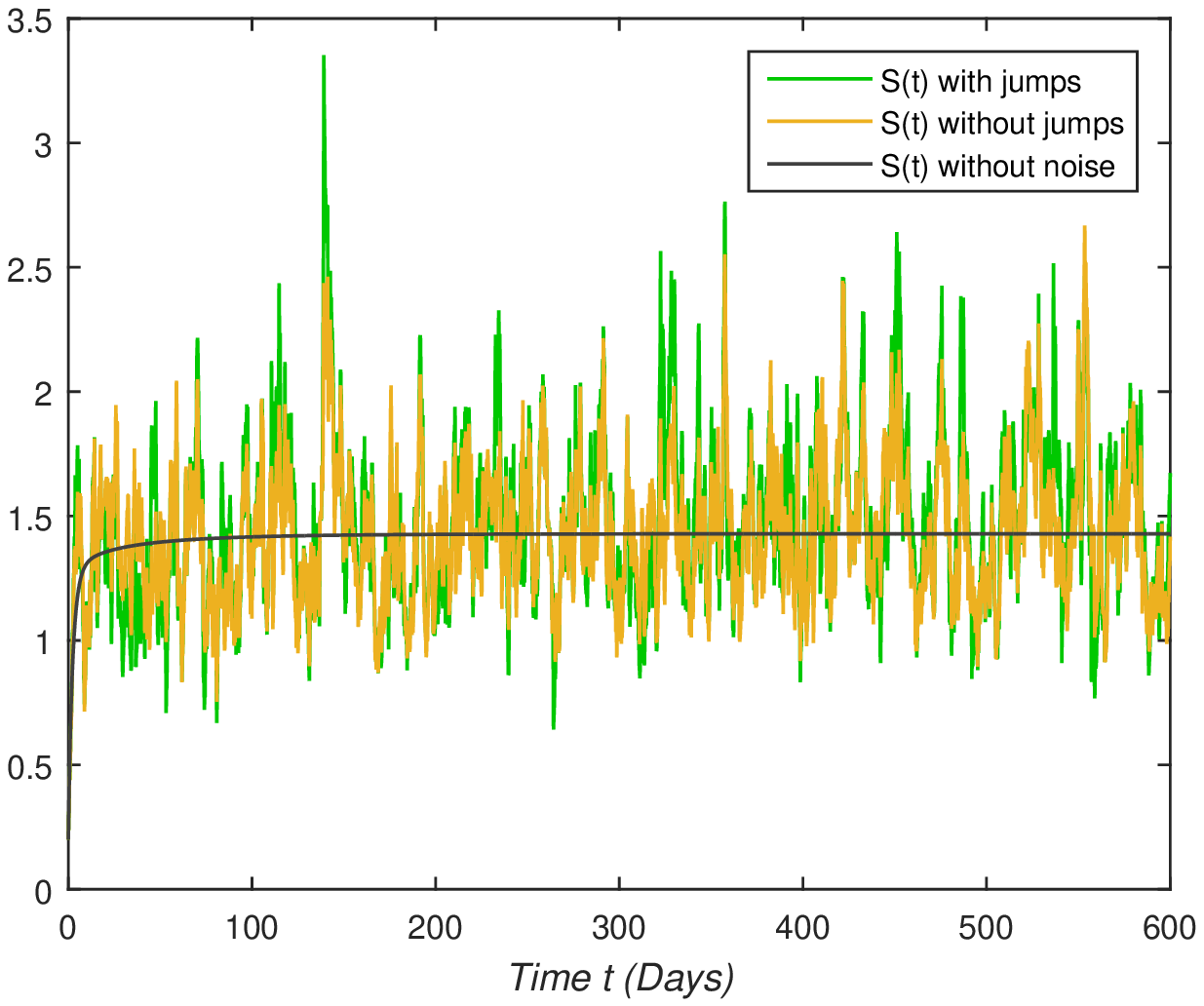} &
\includegraphics[width=3.4in]{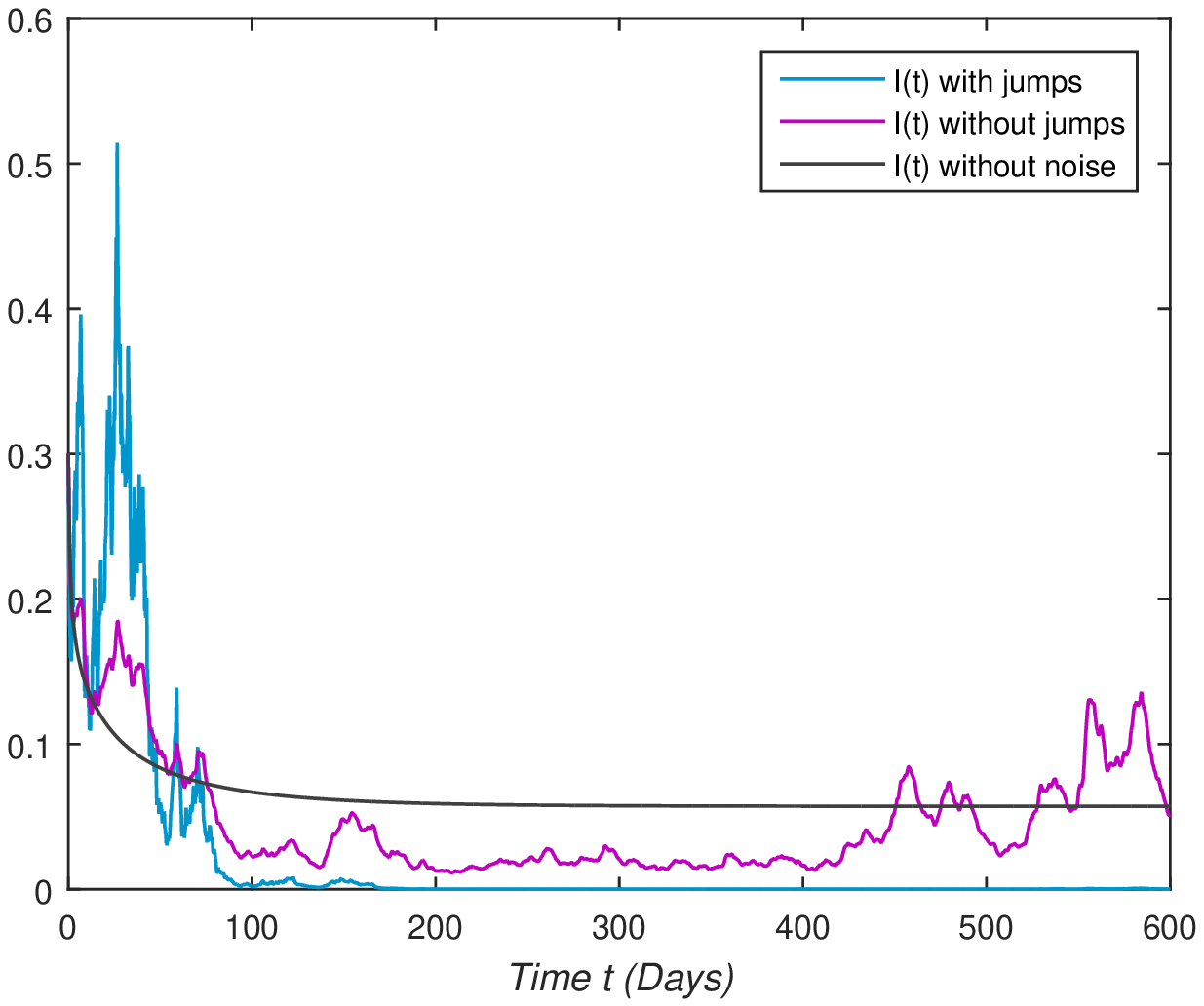}
\end{array}$
$\begin{array}{c}
\includegraphics[width=3.4in]{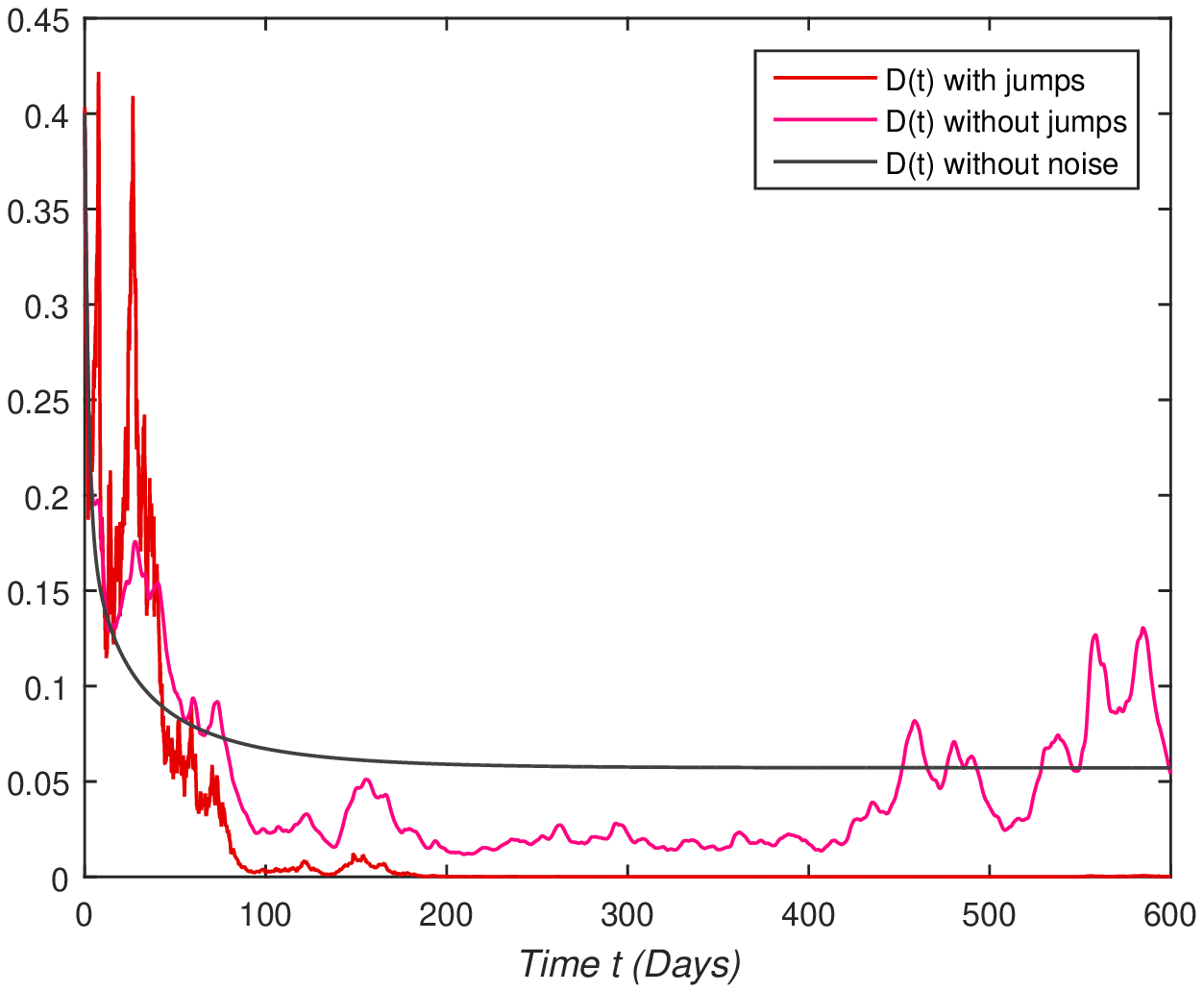}
\end{array}$
\end{center}
\caption{The paths of $S(t)$, $I(t)$ and $\mathcal{D}(t)$ for the deterministic model (\ref{s1d}), the  model (\ref{s12}) with degenerate diffusion, and the SDE-J \eqref{s2}.}
\label{fig3}
\end{figure}
\section{Conclusion and discussion}
Environmental factors and unexpected phenomena have significant impacts on the spread of epidemics. This paper takes into account these two factors. Specifically, we have analyzed a delayed SIR epidemic model that incorporates proportional Lévy jumps. For analytical reasons, we have employed the linear chain approach to \textcolor{black}{transform} the model with a weak kernel case (\ref{s1}) into the equivalent system (\ref{s1d}). After proving the \textcolor{black}{well-posedness} of this perturbed model, we have analyzed its long-term behavior. Under some hypotheses, the main epidemiological findings of our study are presented as follows:
\begin{enumerate}
\item We have given sufficient condition for the extinction of the epidemic. 
\item We have established sufficient condition for the persistence in the mean of the epidemic. 
\end{enumerate}
Compared to the existing literature, the novelty of our work lies in new mathematical analysis techniques and improvements which
are summarized in the following items:
\begin{enumerate}
\item Our work is distinguished from previous works \cite{lev1,lev2,lev3} by the use of the expression $\chi_{1,p}$ which boosts the optimality of our calculus and results.
\item Our study offers an alternative method to the gap mentioned in (Theorem 2.2, \cite{58}). Without using the explicit formula of the distribution stationary $\pi^{\star}(\cdot)$ of $\psi$ (which still up to now unknown), we calculate the following time averages:
 $$\dis\underset{t\to\infty}{\lim}t^{-1}~\int^t_0\psi(s)\textup{d}s\hspace{0.2cm}\mbox{and}\hspace{0.2cm}\dis\underset{t\to\infty}{\lim}t^{-1}~\int^t_0\psi^2(s)\textup{d}s\hspace{0.3cm}\mbox{a.s.}$$
\item In order to find an optimal and good majorization, we have considered the inequality (\ref{optim}) in our analysis without eliminated it (since $\ln(1+x)-x\leq 0$ for all $x>-1$) which differs from the calculus presented in \cite{lev2}.
\end{enumerate}
Generally speaking, our theoretical results indicate that the conditions of extinction and persistence are mainly depending on the magnitude of the noise intensities as well as the system parameters. From numerical simulations, we remark that Lévy jumps affect significantly the long-run behavior of an epidemic. Eventually, we point out that this paper extends the study presented in \cite{da} to the case of Lévy jumps and delivers some new insights for understanding the propagation of diseases with distributed delay. Furthermore, the method developed in this paper can be used to investigate a class of related stochastic models driven by Lévy noise.
\section*{Data Availability}
The theoretical data used to support the findings of this study are included in the article.
\section*{Conflicts of Interest}
On behalf of all authors, the corresponding author states that there is no conflict of interest.
\section*{Funding}
This research did not receive any specific grant from funding agencies in the public, commercial, or not-for-profit sectors.
\section*{Authors' Contributions}
The authors declare that the study was conducted in collaboration with the same responsibility. All authors read and approved the final manuscript.
\begin{figure}[!htb]\centering
\begin{center}$
\begin{array}{cc}
\includegraphics[width=3.4in]{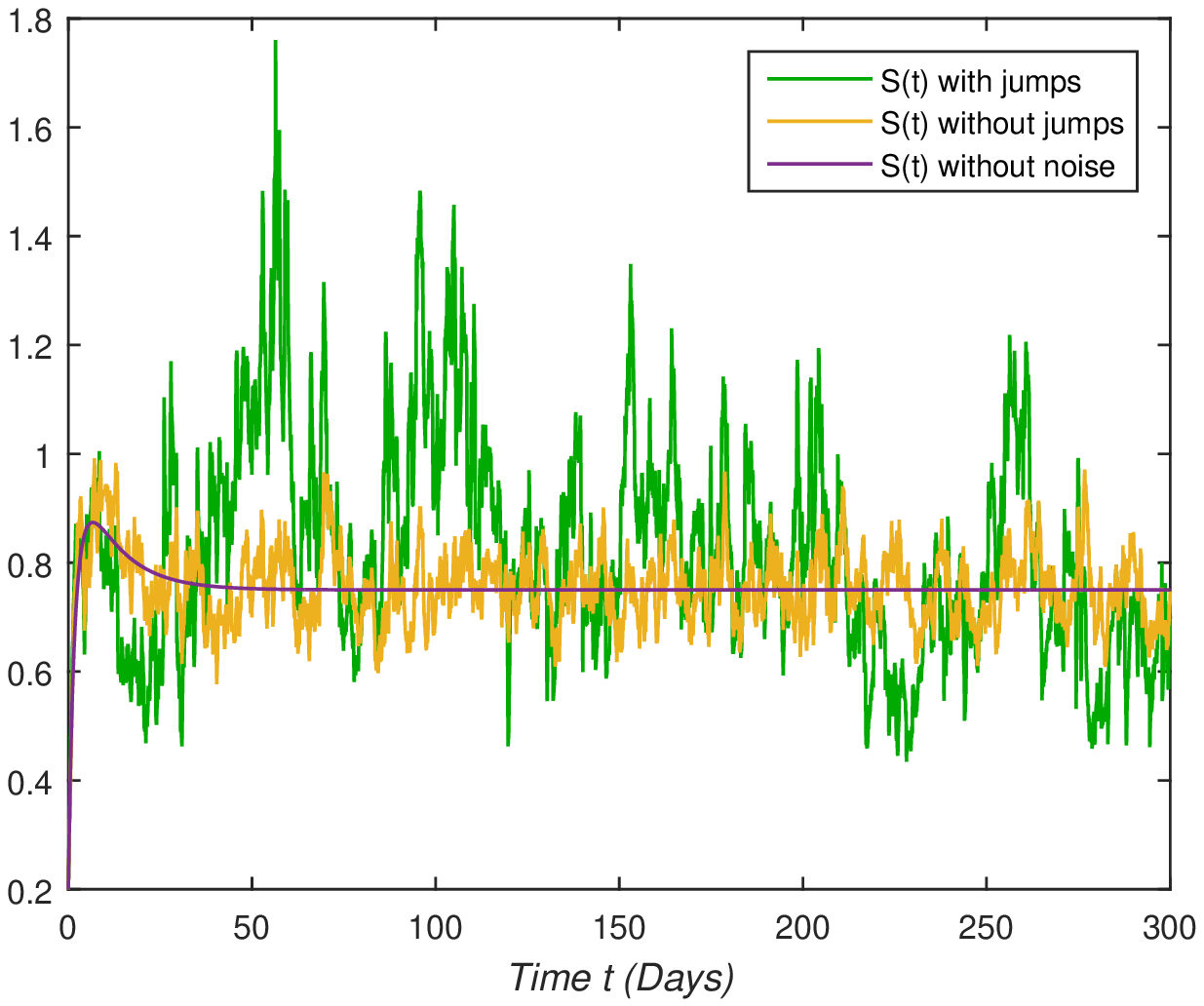} &
\includegraphics[width=3.4in]{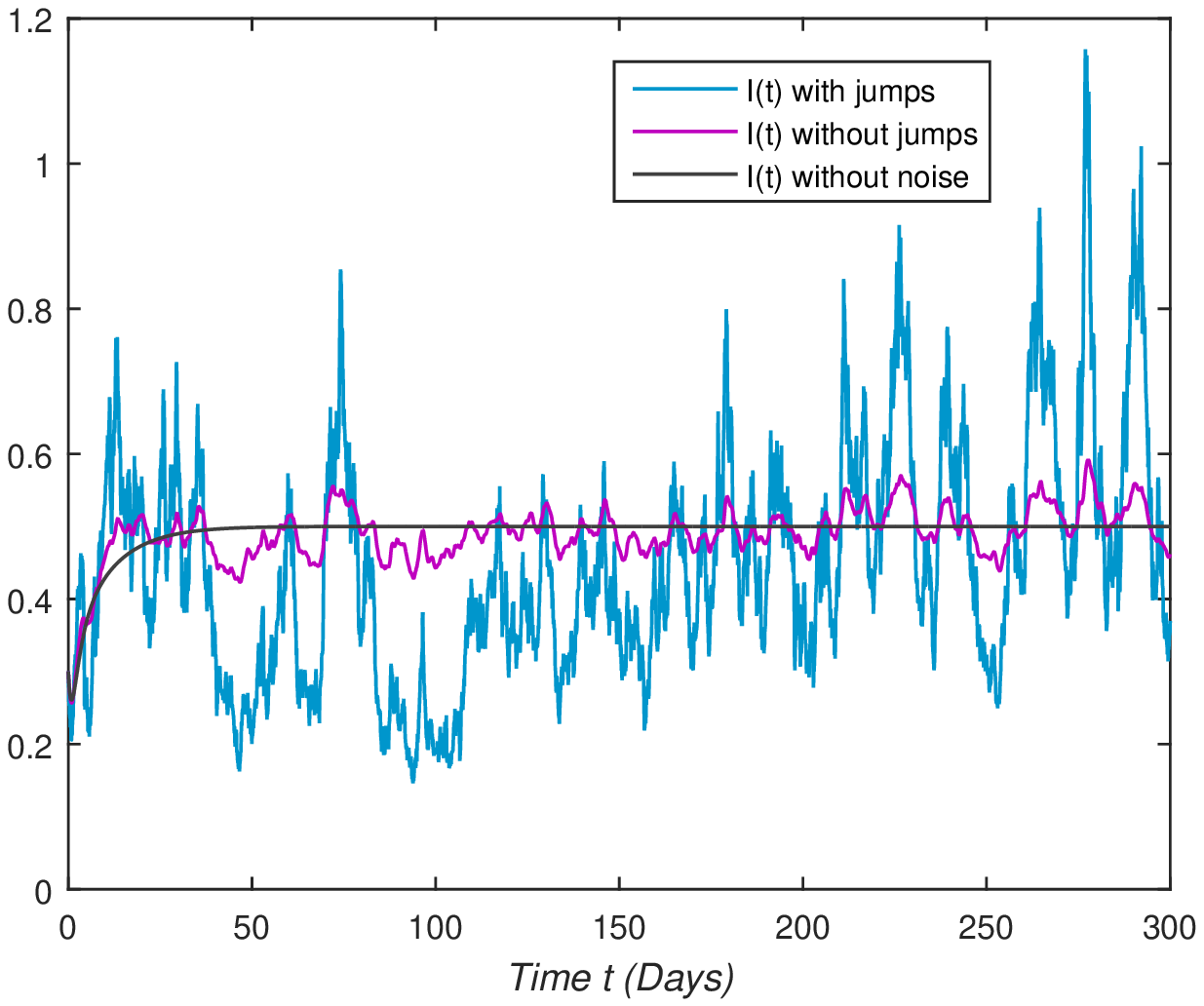}
\end{array}$
$\begin{array}{c}
\includegraphics[width=3.4in]{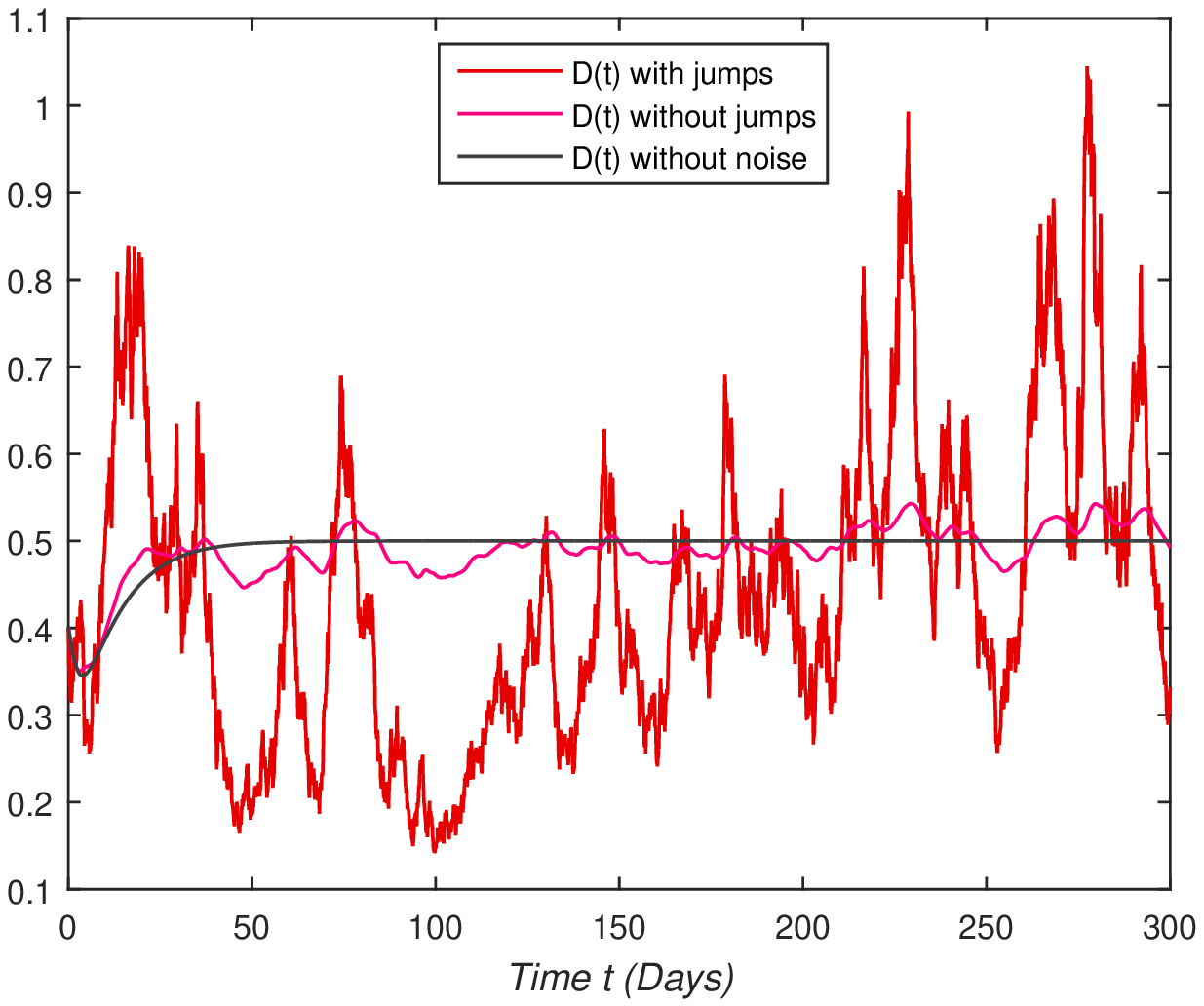}
\end{array}$
\end{center}
\caption{The paths of $S(t)$, $I(t)$ and $\mathcal{D}(t)$ associated respectively to the models (\ref{s1d}),(\ref{s12}) and \eqref{s2}.}
\label{fig4}
\end{figure}
\bibliographystyle{plain}

\bibliography{newbib.bib}
\end{document}